\documentclass[a4paper,twoside,english]{amsart}
\usepackage[T1]{fontenc}
\usepackage[latin9]{inputenc}
\usepackage{babel}
\usepackage{verbatim}
\usepackage{amsthm}
\usepackage{amsmath}
\usepackage{amssymb}
\usepackage[unicode=true,
 bookmarks=true,bookmarksnumbered=true,bookmarksopen=false,
 breaklinks=false,pdfborder={0 0 1},backref=false,colorlinks=false]
 {hyperref}
\hypersetup{
 pdfauthor={V. Mantova and U. Zannier}}

\makeatletter

\pdfpageheight\paperheight 
\pdfpagewidth\paperwidth

\theoremstyle{plain}
\newtheorem{thm}{\protect\theoremname}[section]
 \newcommand\thmsname{\protect\theoremname}
 \newcommand\nm@thmtype{theorem}
 \theoremstyle{plain}

  \theoremstyle{remark}
  \newtheorem{rem}[thm]{\protect\remarkname}
  \theoremstyle{definition}
  \newtheorem*{example*}{\protect\examplename}
  \theoremstyle{definition}
  
  \theoremstyle{plain}
  \newtheorem{lem}[thm]{\protect\lemmaname}
  \theoremstyle{plain}
  \newtheorem{prop}[thm]{\protect\propositionname}
  \theoremstyle{plain}

\makeatletter
  \theoremstyle{definition}
  \newtheorem{my@rem}[thm]{Remark}
  \renewenvironment{rem}{\begin{my@rem}}{\end{my@rem}}
\makeatother

\makeatother

  \providecommand{\examplename}{Example}
  \providecommand{\lemmaname}{Lemma}
  \providecommand{\propositionname}{Proposition}
  \providecommand{\remarkname}{Remark}
  \providecommand{\theoremname}{Theorem}
\providecommand{\theoremname}{Theorem}
 \providecommand{\corollaryname}{Corollary}
\usepackage[left=3.5cm,right=3.9cm,top=3cm,bottom=3cm]{geometry}

\usepackage{mathtools}

\def\Q{{\Bbb Q}}
\def\R{{\Bbb R}}

\def\Z{{\Bbb Z}}
\def\G{{\Bbb G}}
\def\A{{\Bbb A}}

\def\O{{\mathcal O}}

\def\P{{\Bbb P}}
\def\C{{\Bbb C}}
\def\F{{\Bbb F}}
\def\cR{{\mathcal R}}

\def\cE{{\mathcal E}}
\def\cF{{\mathcal F}}

\def\cL{{\mathcal L}}

\def\x{{\bf x}}

\def\w{{\bf w}}

\def\CVD{{\hfill\hfil{\lower 2pt\hbox{\vrule\vbox to 7pt
{\hrule width  5pt\varphifill\hrule}\varphirule}}}\par}

\title[On the Hilbert Property]{On the Hilbert Property and the fundamental group of algebraic varieties}
\author{Pietro Corvaja and Umberto  Zannier}\date{}
\begin{document}
\maketitle

%%% Titoli
%%% On the Hilbert Property and the topology of algebraic varieties
%%%  On simply connected varieties and the Hilbert Property
%%%  On the rational points on covers of algebraic varieties

\noindent {\bf Abstract}. In this paper we  review, under a  perspective which appears different from previous ones,  the so-called Hilbert Property (HP) for an algebraic  variety (over a number field); this is linked to Hilbert's Irreducibility Theorem and has important implications, for instance towards the Inverse Galois Problem. 

We shall observe that the HP is in a sense `opposite' to the Chevalley-Weil Theorem, which concerns  unramified covers;  this link shall  immediately entail the result that the HP can possibly   hold only for simply connected varieties (in the appropriate sense). In turn, this leads to new counterexamples to the HP, involving Enriques surfaces. 

In this view, we shall also formulate an alternative related property, possibly holding in full generality, and some conjectures which unify part of what is known in the topic.
These   predict that for a   variety with a Zariski-dense set of rational points, the validity of the HP, which is defined arithmetically, is indeed of purely topological nature.
Also, a consequence of these conjectures would be a positive solution of the Inverse Galois Problem.

 In the paper we shall also prove the HP for a  K3 surface related to the above Enriques surface, providing what appears to be the first example of a non-rational variety for which the HP can be proved. In an Appendix we shall further discuss, among other things,  the HP also for Kummer surfaces.  All of this basically exhausts the study of the HP for surfaces with a Zariski-dense set of rational points. 

\bigskip

%%%  Introduzione generale: posizione del problema, HP + CW + congetture
%%% HP
%%% Chev-Weil
%%% semplicemente connesse
%%% WWAP
%% esempi punti razionali
%%%  esempi punti interi
%%% discussione congetture

\section{Introduction}

%%%%%%%%%%%%%%%%%%%%%%%%%%%%%%%%%%%%%%%%%%%%%%%%%%%%%%%%%%%%%%%%%%%%%%%%%%%%%%%%%%%%%%%%%%%%%%%%%%%%%%%%%%%%%%%%%%%%%%%%%%%%

This paper is mainly concerned with a kind of  `re-reading,' %`revision', 
so to say, of a circle of issues  around the so-called `Hilbert Property', concerning the set $X(k)$ of $k$-rational points for an algebraic variety $X/k$, where $k$ is a field (which for us shall be  a number field).  The Hilbert Property (abbreviated to HP in the sequel)  is well known to admit  important implications, for instance towards the Inverse Galois Problem (see \cite{SeTGT}, \cite{V}).

 In particular, we shall relate this property with the Chevalley-Weil theorem, which usually, in books or expositions, is dealt with separately from the above.   Further, we shall provide (in the shape of theorems) a number of examples, concerning surfaces,  which we have not found in the literature: we shall consider  new cases in which the HP holds and also shall provide new counterexamples. Moreover we shall  state   some general conjectures relating this kind of property with purely topological aspects of the varieties in question. 

\medskip

As a consequence of this viewpoint, we shall  also re-obtain  some known results,   and of course this article does not claim novelty in these cases.  However,  it seems to us that seldom the links alluded to above %, even if known,   
are presented together within the theory of the HP;  and  it is also in this sense that we hope to contribute to any small  extent. \medskip

For convenience and completeness, we start by recalling in short some  definitions and known  facts about these topics, which shall also serve to %allow us to 
 illustrate simultaneously  our purposes.

\medskip

Throughout,  algebraic varieties shall be  understood to be integral over $k$ and quasi-projective. By {\it cover} of algebraic varieties $\pi: Y\to X$ we shall mean a {\it dominant rational map of finite degree}. In particular, we must have $\dim Y=\dim X$, but $\pi$   need not be a morphism.

%and can be ramified, where by {\it ramified} we mean that % $X$ is normal and
% the branch locus of $\pi$ (i.e. the set of points of $X$ where the fiber has cardinality different from the degree)  is a non-empty divisor.\footnote{So, a blowup of $\P_2$ at a point is not considered here to be ramified above $\P_2$, but ramified above the point.} 

Any possible additional property of $\pi$ shall be specified explicitly. (In particular, in \S \ref{SS.fund}  we shall recall some facts concerning ramification of a cover, a concept especially relevant in this paper.) 

%\medskip

\subsection{\tt The Hilbert Property}\label{SS.HP}  We mainly refer to \cite{SeTGT} (see especially Ch. 3)  for the  theory, limiting ourselves to recall only a few definitions and properties. 

\medskip

\noindent{\tt Definition}. We say that $X/k$ has the {\it Hilbert Property}, abbreviated HP in the sequel, if, given any finite number of covers $\pi_i:Y_i\to X$, $i=1,\ldots ,r$, each of degree $>1$,  the set $X(k)\setminus\bigcup_{i=1}^r\pi_i(Y_i(k))$ is Zariski-dense in $X$.\footnote{In fact, it suffices to assume that this set is always nonempty; indeed, if the set happens to be contained in a proper Zariski closed subset, say  defined by $f=0$, it suffices to consider a further cover with function field $k(X)(f^{1/n})$ for suitable $n$ to lift all remaining points in $X(k)$, thus  making empty the complement with respect to this further cover.}

%%%% ??  dire Y_i , \pi_i definiti su k ? 

\medskip

This property is invariant by birational isomorphism over $k$.  

Any finite union of  a set $ \bigcup_{i=1}^r\pi_i(Y_i(k))$ as above,  together with another finite union $\bigcup_{j=1}^s Z_j(k)$ for proper subvarieties  $Z_j$ of $X$,  is usually called {\it thin} in $X(k)$. The HP may then be restated (as e.g. in \cite{SeTGT}) by saying that {\it $X(k)$ is not thin}. 

\medskip

Of course, for given $X$, this strongly depends on the ground field $k$ (for instance, if $k$ is algebraically closed the HP trivially fails); however, it is preserved by finite extensions, as  in \cite{SeTGT}, Prop. 3.2.1 (but the converse is not necessarily true). 

In the sequel, we shall often  omit explicit reference to $k$, which is supposed to be fixed. 

The relevant fields for the HP are those with `arithmetical' restrictions, so to say, and a field $k$ is called {\it Hilbertian} if the HP holds for some variety $X/k$ (in which case it is known that it holds for any $\P_n$).

 In the present  paper we shall consider  only number fields.

In this case (or more generally when $k$  is finitely generated over the prime field)  the HP is known to hold for projective spaces (Hilbert \cite{H}), for which several proofs are available (see \cite{B-G},  \cite{FJ}, \cite{L}, \cite{Sc}, \cite{SeMW}, \cite{SeTGT}, \cite{V}). Colliot-Th\'el\`ene and Sansuc \cite{CT-S} proved the HP for connected reductive algebraic groups over a Hilbertian field.\footnote{Such algebraic groups are rational varieties, but possibly not over $k$.}   %%% citare CT e gruppi algebrici 
 Other examples may sometimes be constructed starting from these ones: for instance recently it has been shown in  \cite{B}  that if  the HP holds for $X,Y$ it holds for $X\times Y$.  %(as asked  in \cite{SeTGT}, p. 20).  
 In this paper an  example of a non (uni)rational surface with the HP, which appears to be new and not a direct consequence of previous results,  shall be given in Theorem \ref{T.Fermat}. 

\medskip

In the converse direction, a `trivial' reason for the failure of HP  occurs when  $X(k)$ is not Zariski-dense. Of course it can be of the utmost difficulty to establish this: Faltings' celebrated theorems cover e.g. the case of curves of genus $\ge 2$ (and more generally  subvarieties of abelian varieties which are not translates of abelian subvarieties); otherwise,  little is known unconditionally,  and  the Vojta conjectures   give in general   the expected geometrical conditions (see e.g. \cite{B-G}). 

%\medskip

But, concerning the HP itself,  the most interesting failures occur when $X(k)$ may be Zariski-dense and still the HP does not hold. 
In this sense,  the typical examples   that  are usually presented are curves of  genus $1$ over number fields: % for genus $1$  we have an elliptic curve $E$ and
 the (weak) Mordell-Weil Theorem  provides, for every integer $m>1$,  finitely many points $p_i\in E(k)$ such that every $p\in E(k)$ is of the shape $p=p_i+mq$ for a $q\in E(k)$, so $E(k)$ is covered by the sets $p_i+mE(k)$, which are images of covers of degree ($=m^2$) $>1$. 
 
 The same argument applies to any abelian variety (and in turn to any curve of genus $>1$ embedded in its Jacobian, even forgetting about Faltings' theorem).
 
 \medskip

Further examples, in higher dimensions,  occur in the shape of  varieties $X$ mapping nontrivially to an abelian variety;  for smooth $X$,  this happens when the {\it Albanese variety} of $X$  is not trivial, or, equivalently, when the {\it irregularity} $q =\dim H^0(X,\Omega^1_X)$ is positive.  In these cases it may be shown that   the components of  the pullback of multiplication by $m$ produce  covers  of degree $>1$ for suitable $m$, and the failure of the HP for the abelian variety transfers to the variety. (We may agree to say that the examples obtained in this way are not {\it primitive}.) 

In dimension $2$, explicit instances in this direction are ruled surfaces with elliptic base (for which there is a map $f:X\to E$, $E$ elliptic curve, with rational generic  fiber), as for instance the symmetric square $E^{(2)}$ of an elliptic curve $E$, obtained on identifying $(x,y), (y,x)\in E^2$.\footnote{It is easy to see that $E^{(2)}$ is birationally  isomorphic  to $E\times\P_1$.} One may also take surfaces $(E\times F)/G$ where $E,F$ are elliptic curves and $G$ is a finite group acting without fixed points.\footnote{As remarked in \cite{Ha}, there are seven possibilities for a nontrivial such  $G$.} Again, these surfaces admit non-constant maps to elliptic curves.
More generally, quotients of abelian varieties by groups without fixed points also fall in this pattern (it may be shown that the abelian variety is isogenous to a product and that there is a nontrivial map from the quotient  to a factor).

In this paper we  give an example which escapes from all of these constructions:  a smooth surface with a Zariski-dense set of rational points, trivial Albanese variety and failing the HP (see Theorem \ref{T.Enriques}). 

 \medskip

Further, our  aim here is %neither  to discuss these (or other) proofs, nor these examples, but rather 
to point out  certain simple  links of this type of  property with    topological features of algebraic varieties, especially their fundamental group;  in particular, it is  this  connection which suggests and  leads to the said new examples of violation of the HP. 

Simultaneously,  it shall appear that it is equally natural to study a certain modified (weakened)  form of the HP, stated below in \S \ref{SS.modified}. % ??, 
 On the one hand, such property has been implicitly already proved for some varieties for which the HP fails; on the other hand, perhaps this property always holds.
  It is also to be remarked that this different property would admit in principle much the same applications as the HP. 
  
  All of this depends ultimately on ramification, and then we pause  for recalling  some simple concepts in this realm.
 
 %%%%%%%%%%%%%%%%%%%%%%%%%%%%%%%%%%%%%%%%%%%%%%%%%%%%%%%%%%%%%%%%%%%%%%%%%%%%%%%%%%%%%%%%%%%%%%%%%%
 %%%%%%%%%%%%%%%%%%%%%%%%%%%%%%%%%%%%%%%%%%%%%%%%%
 %%%%%%%%%%%%%%%%%%%%%%%%%%%%%%%%%%%%%%%%%%%%%%%%%
 %%%%%%%%%%%%%%%%%%%%%%%%%%%%%%%%%%%%%%%%%%%%%%%%%

\subsection{\tt Fundamental groups and unramified covers}\label{SS.fund}  In this subsection we review some definitions and facts concerning ramification, and we define an algebraic notion,  useful for us,  of being simply connected. (Here we could have relied on the well-known notion of   the profinite completion of the usual  fundamental group; see e.g. M\'ezard's paper \cite{M} for this and more. However for simplicity we have preferred to repeat here  very briefly some relevant constructions and definitions; this shall also allow us to work with singular varieties, and shall make the article more self-contained.) 

\medskip

Let $X/k$ be a (quasi-projective) algebraic variety over a number field $k$. We say that a {\it finite map} $Y\to X$, from a variety $Y$, is unramified (at a point $y\in Y$)  if its differential  (at $y$) is an isomorphism between the corresponding tangent spaces.  Otherwise we say that the map is ramified at $y$, and we call $y$ a ramification point and $x=\pi(y)$ a branch point. 

The set of ramification  points is a proper closed subset of $Y$. Its image in $X$ is generally called the {\it branch locus}. \footnote{However some authors adopt the converse terminology.}

\medskip

This notion does not immediately apply for covers in our general sense, that are merely maps of finite degree. To cope with this, we argue as follows.

We note that given any cover $\pi: Y\to X$ in the previous sense, there exists a normal variety $Y'$ with a birational map $f_Y:Y'\to Y$,   a  normalization    
$ f_X:X^\prime\to X$   and a finite map $\pi^\prime: Y^\prime\to X^\prime$ such that $f_X\circ\pi^\prime = \pi\circ f_Y$. \footnote{The variety $Y'$ may be obtained e.g. by taking on affine pieces the integral closure of the affine ring of $X$ in the function field $k(Y)$, and then by gluing these affine parts.} 

 Then we may define the ramification of the original cover as the ramification of the finite morphism of normal varieties $\pi': Y'\to X'$. This is well defined because the normalization (of $X$)  is unique up to isomorphism.
 
Note that  if e.g. $X$ is smooth, so in particular $X'=X$,  $y\in Y'$ is a ramification point if and only if the number of inverse images $|\pi^{-1}(\pi(y))|$ is strictly smaller than the degree of $\pi$, whereas  otherwise it equals this degree.

% (reduced \footnote{})  {\it ramification divisor}:  forgetting multiplicities, we define it here as  the union  $D$ of components of  codimension $1$ of the closed set of ramification points. We may refer to the ramification divisor (or  branch divisor)  also as the image $\pi(D)$ in $X$. 

Morever, after Zariski's {\it Purity  Theorem},  in case $X$ is  smooth   the branch locus is of pure codimension $1$.
 So in this case we may speak of the ramification divisor in $Y'$ and of  branch divisor in $X$ \footnote{We do not consider here multiplicities; see \cite{B-G} for this.}.

 \medskip

%This divisor is independent of the model in the following sense:if $f_X:X'\to X$ and $f_Y:Y'\to Y$ are birational morphisms and $\pi':Y'\to X'$ is a finite map  with $f_X\circ \pi'=\pi\circ f_Y$ and $D$ is the ramification divisor of $\pi$ while $D'$ is the ramification divisor of $\pi$ then $f_X(D')=D$. 

 \noindent{\tt Definition}. We say that a variety $X/k$ is {\it algebraically simply connected} if    any cover of degree $>1$ of its normalization $X'$  is  ramified (somewhere). 

\medskip 

We note that,  according to our definition, the algebraic simply connectedness must be checked by passing to the normal model.

%The above remarks say that  % this notion is well defined and that 
% it suffices this is verified for some normalisation $X'$ of $X$ (so for instance for a smooth  $X'$). 

\medskip

After fixing an embedding $k\hookrightarrow\C$, we can associate to $X$ the topological space $X(\C)$ of its complex points and so define its topological fundamental group $\pi_1(X)$. \footnote{It was proved by Serre \cite{SeFG} that different embeddings can produce  different fundamental groups; however the properties relevant for us are independent of this embedding, and the same holds for the profinite completion of the fundamental group.}  

In case $X$ is smooth,   if the ramification divisor of a cover is empty, then there is no ramification at all and hence  the morphism $Y'\to X$  defines a topological covering of the space $X(\C)$. 

In general, we have  the following

\begin{prop} \label{P.asc}   If $X$ is normal then $X$ is algebraically simply connected if and only if $\pi_1(X)$ has no subgroup of finite index $>1$.

\end{prop}

\begin{rem} \label{R.asc} We note at once  that the assumption that   $X$ is  normal cannot be omitted, as shown by the example of  the nodal curve $zy^2=x^3+zx^2$ in $\P_2$, obtained by identifying two points of $\P_1$. The normalisation of this curve is $\P_1$, so it is algebraically simply connected in our sense, however the topological fundamental group  of the nodal curve over $\C$ is $\Z$.

We owe to F. Catanese  examples showing that the normality assumption cannot be omitted even for the converse implication; indeed, there are simply connected surfaces whose   normalisation is the product of a line by a hyperelliptic curve.\footnote{One of Catanese's examples consists of a product $\P_1\times E$ where $E$ is an elliptic curve. Taking the quotient by the equivalence relation $(t_1,p_1)\sim (t_2,p_2)$ if and only if $t_1=t_2=0$ and $p_2 = - p_1$ or $(t_1,p_1)=(t_2,p_2)$, we obtain a variety $X$ which may be shown to be simply connected. Its normalization is just the original $\P_1\times E$, whose fundamental group is $\Z^2$.} 

%% PENSARE AL VICEVERSA

%Also, the equivalence in (ii) does not generally hold for normal varieties; an instance is provided by the normal but non smooth model of the Enriques surface of Theorem \ref{T.Enriques}. This is an irreducible  hypersurface in $\P_3$ and therefore the space of its complex points is simply connected by a theorem of Lefschetz;  in particular the condition on $\pi_1$ is satisfied. Also,  it has only isolated points as singularities, and hence is normal (e.g. by a criterion of Serre). However it is not algebraically simply connected, since e.g. a smooth model is the Enriques surface admitting an unramified cover of degree $2$ (or else consider  the finite cover of it appearing in the proof of Theorem \ref{T.Enriques}).
\end{rem}

\begin{proof}[Proof of Proposition]  Let us start by assuming $X$ normal and algebraically simply connected. Suppose by contradiction that $H$ is a subgroup of $\pi_1(X)$ of finite index $d>1$. This corresponds to a connected topological cover $Y$ of $X(\C)$ of degree $d$. By a theorem of Chow  (in the compact case) and of Grauert and Remmert (in the general case) such a topological  cover can be endowed with a (unique) structure of algebraic variety (see \cite{SeTGT}, Ch. 6).  We contend that  $Y$ is an irreducible variety.  This  follows from  a theorem of Zariski (see \cite{Mu}, page 52) asserting that under the normality assumption each point of $X$ is {\it topologically unibranch}  (in the sense of Definition (3.9) in \cite{Mu}, page 43). (Indeed, if $Y$ were reducible, then removing singular points would disconnect it in a neighbourhood of any singular point.) 

But now the cover provided by $Y$ would contradict the property of being algebraically simply connected.

\smallskip

Let us now consider the converse implication.  If $X$ is normal and not algebraically simply connected, let $\pi: Y\to X$ be a finite cover with no ramification.   This  would correspond to a topological cover, and thus to a subgroup of finite index $>1$ in the fundamental group.

This concludes the proof.
\end{proof}

It is important to notice that the notion of algebraic simply connectedness for projective varieties is not a birational invariant, the same as for  the topological notion. (Actually, each projective algebraic variety is birationally isomorphic to a hypersurface of a projective space, and by the Lefschetz hyperplane section theorem, each hypersurface  of $\P_n$, for $n\geq 3$, is simply connected.)

However it is a birational invariant for smooth projective varieties: so
 if one smooth projective model of an algebraic variety is algebraically  simply connected, then every other smooth model is. 
 
 Further, it may happen that a normal model is simply connected, and a smooth model is not:
 an instance is provided by the normal but non smooth model of the Enriques surface of Theorem \ref{T.Enriques}. This is an irreducible  hypersurface in $\P_3$ and therefore the space of its complex points is simply connected by the mentioned theorem of Lefschetz. Also,  it has only isolated points as singularities, and hence is normal (e.g. by a criterion of Serre). However  a smooth model is an Enriques surface, admitting an unramified cover of degree $2$, which shall be implicitly described in our construction below.\footnote{On the contrary, if   a normal model is not simply connected, the same holds for a desingularization, as can be seen on taking  the pullback of a possible cover to the smooth model.}

\medskip

Let us now  state some results in detail.

\medskip

\subsection{\tt Some results} We start with counterexamples.

\medskip

\underline{\tt New failures of the HP}. We do not know of any  prior example % in the literature 
of a variety over a number field with a Zariski-dense set of rational points but without the HP,   which cannot be  reduced as above  to  abelian varieties. We   give such an example in dimension $2$, originated from some considerations below, and we shall prove in  \S \ref{SS.Enriques} the following

\begin{thm} \label{T.Enriques} The (irreducible) surface defined in $\P_3$ by
 the equation $x_0x_2^4+x_1x_3^4=x_0^2x_1^3+x_0^3x_1^2$ has a Zariski-dense set of rational points, has not the Hilbert Property over $\Q$, and does not admit  non-constant rational maps to abelian varieties.
\end{thm}

 This surface is simply connected but not smooth. A smooth model of it  (which is an Enriques surface\footnote{These surfaces were first exhibited  by Enriques: in 1894 he communicated in a letter to Castelnuovo a relevant example of a sextic, which showed that the vanishing of the irregularity and geometric genus were not sufficient for rationality; see \cite{En}. Afterwards, in 1896, Castelnuovo found his celebrated rationality criterion, that the vanishing of the irregularity and the second plurigenus are indeed sufficient. We thank P. Oliverio for these references.} ) is not simply connected, which fits with the viewpoint of this paper.   In \S \ref{SS.K3} we shall add more comments on the nature of this surface (and in Remark \ref{R.noHP}(ii) we shall point out that  the same arguments  work  over any number field.)

\medskip

\underline{\tt New examples of the HP}. Together with this example, we shall prove in \S \ref{SS.K3} the HP for a K3 surface  (which admits the former surface as a 	quotient); namely, the following

\begin{thm}\label{T.Fermat} The surface defined in $\P_3$ by $x^4+y^4=z^4+w^4$ is not (uni)rational (over $\C$) and has the Hilbert Property over $\Q$. 
\end{thm}

We do not know of any   previous example in the literature when the HP is proved unconditionally for a non-(uni)rational (even over $\C$) variety.   Recall that non-rational curves never have the HP, so this  attains the minimal possible dimension in such kind of example. 
(A criterion  for producing varieties with the Hilbert Property has been recently proved in \cite{B}; however, in the case of surfaces this would work only for  rational ones over $\C$, so our example  is new also in this respect.)

\medskip

Together with the HP one can formulate (as in \S \ref{S.conj} 
below) a similar property for $S$- {\it integral} points, rather than rational points. One can surely find examples also for this case, and we shall recall some results for tori $\G_m^n$.

\subsection{\tt The Chevalley-Weil Theorem}\label{SS.CW} A fundamental result linked to the HP, but usually presented in separate discussions, is the Chevalley-Weil theorem. Let us briefly recall (the basic case of) this theorem, abbreviated to CWT in the sequel, which can be seen as an arithmetical analogue of the lifting of maps in homotopy theory. (See \cite{B-G}  for extended statements, e.g. Thm. 10.3.11.)

\medskip

\noindent{\bf CWT}: {\it Let $\pi: Y\to X$ be a(n irreducible)  cover which is also an unramified finite morphism of projective varieties over the number field $k$. Then there is a finite extension $k'/k$ such that $X(k)\subset \pi(Y(k'))$.}

\medskip

 Note that this goes in a direction opposite to the HP, because it asserts that rational points {\it can}  all be lifted to a single cover and number field, albeit  possibly  larger than $k$.

 %We shall sketch a proof later. %, ??  indicate section    .
 
 We do not pause to recall a proof of this, which depends % as usual
on  Hermite's finiteness theorem for number fields of given discriminant. 
Instead, we now give a (simple) proof of the following  proposition,  that  immediately leads to a modified  version of CWT, which allows to maintain the ground field. % (i.e. to take $k'=k$) at the cost of replacing the cover $Y$ by finitely many covers.
\footnote{This version is probably    known, though usually not explicitly mentioned.}  

\begin{prop} \label{P.cw} Let $\pi:Y\to X$ be a cover of degree $>1$ and defined over $k$. Let $k'/k$ be a finite extension and let $T\subset Y(k')$ be a set of points over $k'$ such that $\pi(T)\subset X(k)$. Then there exist finitely many covers $\pi_i:Y_i\to X$, each of degree $>1$, defined   and irreducible over $k$, %and isomorphic to $Y$ over $k'$, 
such that $\pi(T)\subset \bigcup_i \pi_i(Y_i(k))$.
\end{prop}

In particular, we see that if $X(k)$ can be lifted to finitely many covers over $k'$, then it may be lifted already over $k$ at the cost of  adding finitely many  further covers. 

\medskip

Then, putting together the CWT with this   proposition, we   obtain the following  variant (a refinement for what concerns the field of definition):

\medskip

\noindent{\bf Alternative  CWT}: {\it Let $\pi: Y\to X$ be a(n irreducible)  cover of degree $>1$ which is also an unramified finite morphism of projective varieties over the number field $k$. Then there exist finitely many covers $\pi_i:Y_i\to X$  of degree $>1$  such that $X(k)\subset \bigcup_i\pi_i(Y_i(k))$.}

\medskip

Now, note that the number  field $k$ does not increase to $k'$, so   this  expresses something  that is literally  {\it opposite} to the HP;  therefore the HP  is violated for the varieties in question, i.e. those which admit an unramified cover as previously defined. See Theorem \ref{T.sc-HP} below for an explicit statement. 
  
  This fact also gives rise to the new example of violation of the HP alluded to above. 

\begin{proof}[Proof of Proposition]  The proof  is simple. %We assume that $k'/k$ is Galois, which does not affect the statements. 
We let $\widetilde Y$ be the restriction of scalars from $k'/k$ (see \cite{SeTGT}).  If $d=[k':k]$, this is a variety of dimension $d\cdot \dim Y$, isomorphic over $k'$ to $Y^d$. 

The morphism $\pi$ yields a morphism,   denoted in the same way, $\pi: \widetilde Y\to\widetilde X$.  Now, $X$ embeds diagonally as a variety $\Delta\cong X$  into $\widetilde X$ and similarly for $X(k)\subset \widetilde X(k)$. We know that the points in $\pi(T)\subset X(k)$ lift to $T\subset Y(k')$ and hence to $\widetilde Y(k)$. Now it suffices to define  the $Y_i$ as the irreducible (over $k$) components of $\pi^{-1}(\Delta)$. 

Note that  no such component can have degree $1$, for otherwise it would have a rational section (over some finite extension of $k$), and the same would be true for the original cover. %Also, if some such component is not defined over $k$, then it may be accounted only for finitely many $k$-rational points; on adding further finitely many covers we may then assume that these finitely many points can also be lifted.

This concludes the proof.
\end{proof} 

We remark that, if the original cover is Galois, the covers so obtained  are all  isomorphic to $Y$ over $k'$.

The  variant  of the CWT just stated  shows that to violate the HP for $X/k$  it is sufficient to have a finite unramified cover of $X$, i.e. that {\it $X$ is not algebraically simply connected}.  

In practice,  we assert that:

\begin{thm}\label{T.sc-HP} 
Let $X/k$ be a projective algebraic variety with the Hilbert Property. Then  $X$ is algebraically simply connected. Also, the fundamental group of every  normal projective model  of $X$ admits no subgroup of finite index $>1$.
\end{thm}

\begin{proof}
We have already remarked that the HP is a birational invariant. So we may suppose that $X$ is projective, normal,  non algebraically simply-connected and prove that the HP fails. Take a finite unramified cover $Y\to X$ of degree $>1$, possibly defined over a finite extension $k'$ of $k$. Since finite extension of the ground field preserves the HP, as  in \cite{SeTGT}, Prop. 3.2.1., the failure of HP over $k'$ implies the failure of HP over $k$.  
Now, by our refined version of CWT, we obtain finitely many  covers $\pi_i: Z_i\to X$ over $k'$  such that $X(k')=\cup_i \pi_i(Z_i(k'))$, so the HP does not hold for $X/k'$, as wanted.
We have then proved the first part of the theorem.

Taking this into account, the second part is already in part (i)  of  Proposition \ref{P.asc} above.

%Suppose now that $X$ is normal, projective and that the fundamental group (of $X(\C)$) admits a proper finite index subgroup. Then the topological space $X(\C)$ admits a topological finite cover $Y\to X$. By a the mentioned theorem of Chow (\cite{SeTGT}), $Y$ can be given a structure of an algebraic variety, possibly over an extension of the number field $k$. Hence $X$ is not algebraically simply connected and we conclude via the already proved first part of the theorem.
\end{proof} 

This  shows a  basic feature which forbids the HP, and links it with a purely topological property; also,  this  immediately explains the above mentioned examples  related  to abelian varieties. For instance, one also obtains Corollary 4.7 of \cite{B} where it is proved that the only algebraic groups which are of Hilbert type are the linear ones.
But Theorem \ref{T.sc-HP} leads also to new examples of failure of HP (still with a Zariski-dense set of rational points), like  the mentioned one  of Theorem \ref{T.Enriques}:  the projective surface appearing in the Theorem {\it is} simply-connected, but its smooth models are not.  It is also tempting to formulate a %the following 
kind of converse of the above assertion, as we shall do in \S \ref{S.conj} below.

Actually, we shall formulate  more general conjectural statements (also for $S$-integral points). These statements lead to the study of what happens on  considering only ramified covers,  omitting unramified ones: such  omission is not a matter of taste, but has  the good reason that, as we have just observed, the alternative  CWT predicts lifting of all rational points in  all the unramified  cases.  

This motivates the definition of a new property,   including in a sense the HP, and seemingly  with much  the same applications as  the HP. %, even when   the HP itself does not hold. 
See \S \ref{S.conj} for further comments, especially  \S \ref{SS.modified}  for an explicit statement, and  see the papers \cite{C}, \cite{F-Z}, \cite{Z} for some results regarding implicitly  this different property.

\subsection{\tt Weak approximation and the Hilbert Property}  Below we let   $k_v$ denote the completion of the number field $k$   with respect to a place $v$. 

We start by recalling the definition of  the {\it weak-weak approximation property}, denoted WWAP in the sequel:

\medskip

\noindent{\bf WWAP}: {\it A variety $X/k$ has the WWAP if there exists a finite set $S_0$ of places 
of $k$ such that, for every finite set $S$ of places of $k$, $S$ disjoint from $S_0$, the diagonal embedding $X(k)\hookrightarrow \prod_{v\in S} X(k_v)$ is dense.}

\medskip

In other words, there are so many rational points on $X$ to  approximate simultaneously  within $S$ to any finite set of $k_v$-rational points. There is an analogous notion for integral points.  Also, we speak of  a {\it weak approximation property}  WAP   if $S_0$ may be taken to be empty.\footnote{There are also concepts of {\it strong} a. p.  - different from WAP only for integral points -  and of {\it hyper-weak} a. p., this last one being introduced by Harari, see \cite{Ha3}.} 

\medskip

It is not difficult to see that  for smooth   varieties   (not necessarily complete)  these  properties are invariant by birational isomorphism over $k$, as happens for  the HP.

Actually,   a precise link with the HP was established by Colliot-Th\'el\`ene and Ekedahl, who  observed  that  {\it the WWAP implies  the HP}:  see \cite{SeTGT}. (In fact, those  arguments   show a bit more, i.e.  that it is sufficient that `most' local points can be approximated.) %\footnote{We   do not know whether HP does  imply WWAP, though we shall point out it does not imply WWA already in the case of surfaces.}

It shall appear as a consequence of conjectures below  that {\it $k$-unirational varieties have the HP}.  Through  the result just mentioned, this would follow also from a conjecture of Colliot-Th\'el\`ene asserting that {\it a unirational variety has the  WWAP}.

The implication established by Colliot-Th\'el\`ene and Ekedahl, combined with Theorem \ref{T.sc-HP}, results in particular in the following further  implication:

\medskip

\centerline{\it If the projective variety $X/k$ has the WWAP then $X$ is algebraically simply connected.}

\medskip

This conclusion  appears (in slightly different form) already in the 1989 paper \cite{M} by Minchev, where it is proved independently of the above circle of observations (but at bottom with similar arguments). See also Harari's paper \cite{Ha2} for a survey discussion presenting also several examples related to (weak) approximation properties. The paper \cite{Ha} by Harari also discusses  interesting relations between obstructions to weak approximation and the fundamental group; these go in a direction similar to the present paper, even if  the HP is not directly considered.

These results represent   instances, different from e.g. the well-known ones involving zeta functions,  of how an arithmetical aspect may happen to reveal and actually prove a topological one. 

\medskip

Concerning  mutual implications among these properties, apart from the mentioned ones, we shall discuss in \S \ref{SS.HW}  below  a cubic surface, unirational but not rational over $\Q$,  considered already by Swinnerton-Dyer in \cite{SwD2}. He proved that it has  not the WAP; we shall reproduce his argument and slightly expand it for a similar surface.  We shall also prove the HP for such a surface, so  we can conclude that

\begin{thm}\label{T.HP-WAP}  The Hilbert Property for a surface over $\Q$ does not generally imply the Weak Approximation Property for its smooth part.
\end{thm}

 \noindent (Note that the analogous statement  for curves would not be true, since the only curves with the HP are rational and thus their smooth part has  the WAP.)  However we do not know whether the HP generally  implies at least  the WWAP; we tend to believe this should not be the case, but for this we have neither a proof nor real evidence. 

The Appendix 2 below shall discuss  the WWAP for the above cubic surface, essentially proving (a slightly different form of) it: see Theorem \ref{T.A}.\footnote{Actually, an independent and more general proof of the WWAP for this surface appears already in the paper \cite{CT-S-SD}; we thank Colliot-Th\'el\`ene for this (and related) reference(s).} 
 (This is  also sufficient to yield a different proof of the HP for that surface.) 
 
 \smallskip

Other examples in this direction are possible for integral points. We just pause for a few words on the celebrated  {\it Markoff-surface}  $x^2+y^2+z^2=3xyz$, in which the nonzero  integral points are Zariski-dense and all in a single orbit of $(1,1,1)$ by the  (discrete) group of automorphisms, generated by $(x,y,z)\mapsto (x,y,3xy-z)$ and the permutations of coordinates. The problem of WAP for integral points on  this surface is open, but Bourgain, Gamburd and Sarnak \cite{BGS} proved much in this direction, showing in particular  that for `almost all' primes (in a strong sense) we may approximate by integral points  any nonzero point over $\F_p$.   (Their methods  should be amply sufficient for instance  to derive the HP for integral points of this surface.)  

\medskip

In the next section we shall raise some questions and conjectures originated from the links that we have alluded to. In \S \ref{S.Ex} we shall give detailed proofs of the mentioned results, which concern  surfaces of K3 and Enriques type. In the Appendices we shall sketch proofs of other assertions, for Kummer surfaces and relations with approximation properties. In view  of the Bombieri-Lang conjecture  that surfaces of general type should never have  a Zariski-dense set of rational points, all of this basically exhausts the study of the HP for surfaces with a Zariski-dense set of rational points.

\section{Some questions, conjectures and variations}\label{S.conj}

In this section, in the direction of what we raised before,  we put forward some questions and  tentative conjectures, and also state some further results. We shall consider also the case of integral points on affine varieties. 

We start with   questions.

\subsection{Questions and conjectures} \label{SS.qc}  We have pointed out that a projective variety non algebraically simply connected  cannot have the Hilbert Property and we start by asking a converse of this:

\medskip

\noindent {\bf Question-Conjecture 1}: {\it Any  smooth algebraically simply connected  projective variety (over  $k$) with a Zariski-dense set of  rational points has the Hilbert Property (over $k$)}.
  
\medskip

Note that the smoothness assumption cannot be omitted, a relevant example coming  (once more)  from the surface in  Theorem \ref{T.Enriques}.

A proof of  this statement, if at all true, would be most probably very difficult.  

Since the HP is a birational invariant,  any projective variety with a smooth model which is simply connected  should have the HP  as soon as its rational points are dense. 
This  would include for instance unirational varieties (over $k$) since after a result of Serre \cite{Se} any smooth model of a unirational variety is   simply connected and of course their $k$-rational points are Zariski-dense.
 By a more general result   - see \cite{D}, Corollary 4.18  - this would also imply that a   rationally connected variety over $k$ has the HP \footnote{This case is also linked with other questions concerning the Brauer Manin obstruction: it is not known whether for rationally connected varieties this is the only obstruction to weak approximation: if this   was the case, then HP for smooth rationally connected varieties would hold; we than Borovoi for explaining us this link}. 
 
  We also remark that in turn this would imply a positive answer to  the Inverse Galois problem. Indeed,  as is well known (see \cite{SeTGT} and \cite{V}), there is a strong  link between the HP and the Inverse Galois Problem, provided by the method introduced by E. Noether of taking a quotient e.g. of $\P_n$ by the action of  a linear representation of a given finite group $G$ (for instance by permutation of coordinates), obtaining a unirational variety. %: given a finite group $G$, we can let $G$ act (over $\Q$) on a projective space $Y$, or on a product $Y$ of projective spaces (e.g. by permuting coordinates, after realizing $G$ as a permutation group). Let $X$ be the quotient variety (in general singular). One can look for a point on $X(\Q)$  whose pre-image in $Y$ is irreducible over $\Q$. This would give by specialization a realization of $G$ as the Galois group of a finite extension of $\Q$. Of course, the points of $X(\Q)$ coming from rational points on $Y$ would be useless in this respect, and more generally we have to discard the points coming from rational points on intermediate covers. This leads precisely to the problem of the validity of the HP on such an $X$ (which is unirational).

\medskip

However, this statement  looks very optimistic (only little  evidence is provided by Theorem \ref{T.Fermat}), and as in some conjectures of Vojta  it is perhaps necessary to weaken the conclusion, and  to allow (at least) an enlargement of the number field in order to ensure  the Hilbert Property. 

To appreciate the strength of the statement, we note  that  the above version  would have striking (unknown) consequences as the following one:

 {\it Let $A$ be an abelian variety over $\Q$ with a Zariski dense set of rational points. Then there exists a non-rational  point $p\in A$, defined over some  quadratic field such that $-p$ equals the quadratic conjugate of $p$.} 

This is obtained on applying the statement of Question-Conjecture 1 to the quotient $X:=A/{\pm 1}$, which may be checked to be algebraically simply connected (it is $\P_1$ when $A$ is an elliptic curve). Since $A$ is a double cover of $X$, the statement predicts rational points on $X$ which do not come from $A(\Q)$, and this translates in the existence of $p$. Actually, one would get a Zariski-dense set of such points.

Already the cases when $A$ is a product of several elliptic curves  lead to arithmetical questions such as: 

{\it Let $f_1,\ldots,f_r\in\Q[x]$ be $r\geq 2$ cubic polynomials such that all curves $y^2=f_i(x)$ are elliptic with infinitely many rational points. Is it true that necessarily there are rational numbers $u_1,\ldots,u_r$ such that  no $f_i(u_i)$ is a square but all products $f_i(u_i)f_j(u_j)$ are squares in $\Q^*$?}

 A positive answer would follow from the above.

We do not know how to treat this problem in general, but in Appendix 1 below we shall give a positive answer for the cases $\dim A=2$. In these cases the surfaces $X$ which appear  are  called {\it Kummer surfaces}, a classical object of study.

\medskip

Anyway, apart for these instances,  we have no real evidence even for  the validity of these hypothetical statements (not to mention the full conjecture). On the other hand, at least the  said implication  for the quotients $A/\pm 1$ may be obtained  on enlarging the ground field, for which  we sketch a possible argument: on taking a dominant map, say of degree $d$, from $X$ to $\P_n$ ($n=\dim X$), just by Hilbert's theorem we may find rational points on $\P_n$ which lift to points  $p$ of degree   degree $2d$ on $A$, and degree $d$ on $X$. If $k$ is the field generated by the image  of $p$ in $X$, then $p$ is defined over a quadratic extension $k'$ of $k$, but not over  $k$. By taking multiples of $p$,  in general we  obtain a Zariski-dense set of such points on $A$ (with the same $k,k'$), which give what is needed. (These multiples shall be Zariski-dense for a suitable choice of  the starting rational point on $\P_n$.)

\medskip

%The smoothness assumptions in the above Question  are relevant, as shown from the singular model of the Enriques surface appearing in Theorem \ref{T.Enriques}, which is simply connected but does not satisfy HP.

In any case, forgetting the hypothesis on simply connectedness, we can also extend  the statement to all varieties, asking for instance the following variant.

\medskip

\noindent {\bf Question-Conjecture 2}: {\it Let $X/k$ be  a  variety with a Zariski-dense set of rational points.   If $\pi_i:Y_i\to X$, $i=1,\ldots ,m$, are finitely many covers of degree $>1$,  each with non-empty  ramification, then  $X(k)-\bigcup_{i=1}^m\pi_i(Y_i(k))$ is still Zariski-dense in $X$.}  % In particular, no smooth model of $X$ is simply connected}. 

%\noindent {\bf Question-Conjecture 2}: {\it Let $X/k$ be  a  variety with a Zariski-dense set of rational points.   If $\pi_i:Y_i\to X$, $i=1,\ldots ,m$, are finitely many covers of degree $>1$,  and $X(k)-\bigcup_{i=1}^m\pi_i(Y_i(k))$ is not Zariski-dense in $X$, then there is an  index $i\in\{1,\ldots m\}$ and smooth models $Y_i'$ of $Y_i$, $X'$ of $X$ such that  the corresponding map  $\pi_i': Y_i'\to X$ is  an unramified cover. In particular, no smooth model of $X$ is simply connected}. 

%\smallskip

%By our remarks in the introduction, we could easily formulate an equivalent statement by requiring that $X$ is smooth projective and the covers $\pi_i:Y_i\to X$ are finite maps. In that case the conclusion would be that at least one cover is unramified.

\medskip

Again, a weaker and more plausible version is obtained  by allowing a number field larger  than $k$ in the conclusion, and/or also by adding   more stringent requirements on the branch locus, or else the conclusion could maybe hold generally only after  bounding  the dimension. 

In the case of curves, Question-Conjecture 2 is settled: the crucial cases arise from curves of genus zero (solved by  Hilbert Irreducibility Theorem) and curves of genus one (which follows from   Faltings' Theorem\footnote{Historically, the case of curves of genus one was  settled previously by N\'eron, \cite{Ne}, via a method based on heights estimates on abelian varieties, subsequently much improved  by Mumford; still another proof for this case can be deduced from the results in \cite{Z}.}). 

Apart from the above remarks, some evidence for these statements comes from the celebrated conjectures of Vojta, which, very roughly speaking, predict that (possibly at the cost of a finite extension of the ground field) the distribution of rational points is governed by the `bigness' of the canonical class: {\it the bigger $K_X$ is, the smaller chance we have to find many rational points on $X$}. (See \cite{B-G}, Ch. 14,  for a discussion of several  of these conjectures.) 
 
 Now, if $X$ is (smooth and) simply connected, any cover $\pi:Y\to X$ of degree $>1$ has ramification, and the ramification divisor $R$ (considered now with appropriate multiplicities, see \cite{B-G}, p. 473) contributes to $K_Y$, for we have $K_Y=\pi^*(K_X)+R$. Hence, if $R\neq 0$, and especially of $R$ is big,  we expect to find less rational points on $Y$ than those which would be accounted if most rational points on $X$ would lift.
 
 Actually, this principle is made much  more explicit in Manin's conjectures on the distribution of rational points on Fano varieties, i.e. those with $-K_X$ ample. After some variants of Manin's conjecture, one could expect to find `more' points on $X(\Q)$ then those coming from $Y(\Q)$ (or from intermediate covers).

Perhaps a version of Manin's conjecture for {\it singular} Fano varieties, as formulated by Batyrev, Manin and Tschinkel (see e.g. \cite{BM}, \cite{BT}) might provide  the HP for suitable quotients of projective spaces,  so   implying a positive solution to the Inverse Galois Problem by the  approach of Noether mentioned above. 
Some steps in this direction have been done by T. Yasuda in \cite{Y1}, \cite{Y2}.

\subsection{\tt A modified HP} \label{SS.modified} The property implicit in  the {\it Question-Conjecture 2}   is a modified natural form of the HP, which we shall call {\it Weak Hilbert Property} abbreviated WHP, i.e.:

\medskip

\noindent{\tt Definition}. We say that a normal variety $X/k$ has the {\it Weak Hilbert Property}  WHP if, given finitely many covers  $\pi_i:Y_i\to X$, $i=1,\ldots ,m$, each  ramified above a non-empty divisor,  $X(k)-\bigcup_{i=1}^m\pi_i(Y_i(k))$ is Zariski-dense in $X$.

\medskip

Namely, it is the same as the HP,  except that we consider only covers ramified (above a divisor). And hence, in particular,  for algebraically simply connected varieties it coincides with the HP. 

In fact, we have seen that there is a very good reason to {\it disregard  unramified covers},  because the (varied) CWT shows that as soon as we have such a cover  then all rational points can be lifted to finitely many covers.  

As expressed by the above  Question-Conjecture 2, this new property WHP, differently from the HP,  could possibly hold for {\it all varieties with a Zariski-dense set of rational points}. 

For instance, it holds for all  curves with infinitely many rational points ({\it a posteriori} of genus $\le 1$), since   it holds for rational curves and since a ramified cover of a curve of genus $1$ has genus $>1$ and one can apply Faltings' theorem (or even a previous  estimate on heights by Mumford).    For surfaces, the results of this paper combined with the results of \cite{Z} for  covers of some abelian surfaces and with the Bombieri-Lang  conjecture for surfaces of general type, also provide substantial evidence of a general validity. 

Regarding abelian varieties of any dimension $>1$, unfortunately there are no general results for rational points on {\it covers} of them as strong as the results for {\it subvarieties}   (although the mentioned Vojta's Conjecture predicts degeneracy of rational points on each ramified cover of an abelian variety).  The issue  of covers  is treated in \cite{Z} with different methods compared to the use of  Faltings' or similar  theorems. This is sufficient to prove the WHP   for products of elliptic curves with the (unessential)  restriction that we deal with a Zariski-dense {\it cyclic}  group of rational points.  Those  methods in principle should extend to any abelian variety, provided certain results on the Galois action on torsion points are available.\footnote{This happens in considerable generality, as shown e.g. by Serre, after Faltings' proof of certain Tate conjectures.}

\subsection{\tt Integral points}   One can formulate in a rather obvious way analogous statements for $S$-integral points, in which the variety $X$ can  be non complete, and the ($S$-) integral points are meant e.g. with respect to a compactification.\footnote{They may be defined as the points which do not reduce to the divisor at infinity for any prime outside $S$.}  In this case the analogue of the HP is that {\it not all $S$-integral points on $X$ can be lifted to rational points of finitely many covers of degree $>1$}.\footnote{It is immaterial that we require that the points are integral also in the covers, since one can change anyway slightly a cover to get a finite map and ensure integrality of rational points sent to integral ones.}   

Now, there is  a version of the Chevalley-Weil Theorem  true for $S$-integral points. We state it, denoting by $\O_S$ the ring of $S$-integers of $k$ and by $X(\O_S)$ set of $S$-integral points of the affine variety $X$:

\medskip

\noindent{\bf CWT for integral points}: {\it Let $\pi: Y\to X$ be a(n irreducible)  cover which is also an unramified finite morphism of affine  varieties over the number field $k$. Then there is a finite extension $k'/k$ such that $X(\O_S)\subset \pi(Y(k'))$.}

\medskip

 Much the same considerations as above apply for this case, and one is led to a WHP on considering only covers with a divisor of ramification not contained in the divisor at infinity. 

A general version of Question-Conjecture 1 could be

\medskip

{\bf Question-Conjecture 1bis}. {\it Let $X/k$ be a smooth affine  variety topologically simply connected  and such that the set of $S$-integral points are Zariski-dense. Let $\pi_i: Y_i\to X$ be morphisms of finite degree from algebraic varieties $Y_1,\ldots,Y_n$. Then $X(\O_S)-\bigcup_{i=1}^{n}\pi_i(Y(k))$ is Zariski-dense.}

%%sulla congettura 1bis, bisogna stare attenti ai modelli. Infatti ai noi interessa che X sia semplicemente connessa, indipendentemente dal fatto che lo sia un modello proiettivo.Forse potremmo dire: X smooth quasi-projective variety, (topologically) simply connected' cos siamo sicuri.

\medskip

We can naturally formulate also an analogue  of Question-Conjecture 2. For brevity we do not repeat this, 
and only recall that  some results in this direction are proved in \cite{C}, \cite{F-Z}, \cite{Z} for multiplicative tori $X=\G_m^n$.
In this toric case   the $S$-integral points are those with $S$-unit coordinates, and they form a finitely generated group. In the quoted papers  it is shown in particular that a Zariski-dense subgroup of integral points cannot be lifted to rational points  of a finite number of ramified covers, confirming the statement of Question-Conjecture  2 (i.e. the WHP) for these cases.   %  In \cite{Z}, the case of a product of elliptic curves is also treated, with a similar conclusion for Zariski-dense cyclic subgroups (but the methods should extend to arbitrary subgroups and also to more  general abelian varieties).  %% dare esempi in cui viene dimostrato  lavori C e FZ, Z. o fore questi nella sez. finale.

\subsection{\tt The Hilbert Property and Nevanlinna Theory}  We conclude this introductory part by mentioning that these topics admit natural analogues in the context of Nevanlinna Theory (where deep links  with Diophantine Geometry have been recognised since long ago, especially by Vojta, see \cite{Voj} and \cite{B-G}). For instance, the analogue of Chevalley-Weil is well known: since $\C$ is simply connected, given an unramified cover of algebraic varieties $Y\to X$ over $\C$, every holomorphic map $\C \to\ X$ lifts to $Y$. 

Concerning  the HP, an analogue of Question-Conjecture 1 could be   as follows:

{\it  Let $X$ be a simply connected smooth  projective algebraic variety over $\C$ with an analytic map $f:\C\to X$ with Zariski-dense image.  Then for every cover $Y\to X$ (in the present sense) of degree $>1$ there is an analytic map $g:\C\to X$ that does not lift to $Y$.}

Similarly, there are analogies also with other issues raised in this paper.

\section{Proof of main assertions}\label{S.Ex}

\subsection{A non-rational surface with the Hilbert Property and the proof of Theorem \ref{T.Fermat}}\label{SS.K3}

In this subsection we provide an example of a non-rational surface with the Hilbert Property, proving in particular Theorem \ref{T.Fermat}. 

It seems that such examples, where one  can actually prove the HP and  not merely fit it into conjectures,  are not so common: on the one hand we know from the above that we must start with a simply connected variety, on the other hand we have to provide `many' rational points.\footnote{It is worth noticing that, somewhat in the converse direction,  no example seems to be  known of a  smooth simply connected  projective surface where the rational points can be proved to be not Zariski-dense; this  should be expected e.g. for smooth hypersurfaces in $\P_3$ of large enough degree (which are of general type). Instead, for integral points on affine surfaces such an example may be found in the paper \cite{CZ2} by the authors: see Thm. 3 and Cor. 2 therein.}

\medskip

Our example is furnished by the Fermat quartic smooth surface (K3) defined in $\P_3$ by 
\begin{equation}\label{Fermat}
F:\quad x^4+y^4=z^4+w^4.
\end{equation}

It is well known that this is not (uni)rational (because e.g. the canonical  class vanishes, hence it has global sections) and that it  is simply connected, because e.g. it is a smooth complete intersection of dimension $>1$. 

Also, it has a Zariski-dense set of rational points; this is due to   Swinnerton-Dyer  \cite{SwD}.\footnote{He actually proved the density of rational points in the euclidean topology, inside the set of real points}  (Being simply connected and with a Zariski-dense set of rational points, the HP  is actually predicted by the conjectures above.)

We reproduce a proof of this density here, borrowing from  Swinnerton-Dyer's  construction, for completeness and usefulness for  the sequel. 
%For convenience we shall  use an affine equation \begin{equation*}1+x^4=y^4+z^4.\end{equation*}

\medskip

This surface $F$ contains forty-eight  lines, of which eight are defined over $\Q$, and  given %in the affine part 
by  $x=\pm z, y=\pm w$, or $x=\pm w, y=\pm z$. 
A first idea is, for one of these lines, to consider the pencil of planes through this line and intersect it with $F$; this shall produce a curve which is a union of the line with a cubic in the plane, which turns out to be generically smooth. Now, any other of the lines which does not intersect the former one shall intersect the plane in one point  necessarily on the cubic. This yields a section of the pencil of cubics (i.e. a rational map from any point of $\P_1$ to the corresponding cubic); in turn, choosing this point as an origin on the cubic, shall yield an elliptic curve. By changing this last line we shall obtain another section and thus a point on the elliptic curve, distinct from the origin. \footnote{Here we could also use the `tangent  process', i.e. intersecting the cubic with the tangent at the previous point;  or else we could  intersect the cubic with the line. However this last method, though sufficient for our purposes,  would produce rational points only for a `thin' subset of the pencil.}   It turns out that this point is not (identically) torsion. 

Now, restricting the pencil to rational points on $\P_1$, we obtain a family of elliptic curves defined over $\Q$, with a rational point which can be torsion at most for finitely many values (by a theorem of Silverman or a direct argument). This provides a Zariski-dense set of rational points.

\medskip

To be explicit, say that we start with the line $L:x=z,y=w$, with corresponding pencil of planes defined by $\Pi_\lambda: w-y=\lambda(x-z)$, $\lambda\in\P_1$.   We find for the cubic the (further) equation
\begin{equation*}
%E_\lambda:\quad 
x^3+x^2z+xz^2+z^3=4\lambda y^3+6\lambda^2y^2(x-z)+4\lambda^3y(x-z)^2+\lambda^4(x-z)^3.
\end{equation*}

Call $E_\lambda$ this cubic (understood on the said plane). It is singular only  if $\lambda(\lambda^8-1)=0$ (or $\lambda=\infty$), and we shall tacitly disregard  these points.

  Note that the $E_\lambda\cup L$  are the fibers for the rational map $\lambda:F\dashrightarrow \P_1$, given by $\lambda:={w-y\over x-z}$. Actually, this map is well defined on $F-L$ \footnote{In fact, if we define $\lambda$ by the same formula on the whole $\P_3$, then it is not defined at any point of $L$. However, its restriction to $F$ can be continued to a regular map on the whole surface.} with fibers $E_\lambda -L=E_\lambda -\{(\eta:1:\eta:1):\eta^3=\lambda\}$, and of course restricts to a regular map on the whole $E_l$, provided $l$ is such that  this last curve is smooth. 

\medskip

 If we intersect  the plane $\Pi_\lambda$ with the different  line $L': x=-z,y=-w$, we get the point  $(1:-\lambda:-1:\lambda)\in E_\lambda$. As above, we equip $E_\lambda$ with a structure of elliptic curve by prescribing this point as the origin. 

Intersecting now the same plane with the line $x=w, y=-z$, we get the new point $(\lambda +1:1-\lambda:\lambda-1:\lambda+1)\in E_\lambda$.

We omit here the verification that the difference of the two points is not torsion identically in $\lambda$, and refer to \cite{SwD}; we only add that  this may be done by using either  specialization at some $\lambda\in\Q$ (and then any of the usual arithmetic methods) or also on considering functional heights, on viewing $\lambda$ as a variable.

In particular, by a well-known result of Silverman, the values of $\lambda$ such that this point is torsion on $E_\lambda$ are algebraic and have bounded height, hence the rational values of this type are finite in number. This implies that $E_l(\Q)$ has positive rank for all but finitely many $l\in\Q$, proving that $F(\Q)$ is Zariski-dense in $F$.  %See   \cite{SwD}  for explicit verifications, using in practice the Lutz-Nagell theorem and avoiding the appeal to Silverman's result.

\medskip

\begin{rem} The paper \cite{SwD} of Swinnerton-Dyer gives a Weierstrass equation and uses in practice the Lutz-Nagell theorem for any rational value of $\lambda$,   avoiding  the appeal to Silverman's result. Also, it does 
not only verify that the rational points are Zariski-dense in $F$, but proves  that they are dense in the real topology. This is a `piece' of the WAP  for $F$. On the other hand, we do not know whether $F$ has the WAP or even the WWAP. This is studied by Swinnerton-Dyer also in the more recent paper \cite{SwD3}, where a double elliptic fibration is used, as we  shall also do for our (different) purposes. 
\end{rem}

\medskip

For  future reference, we note that starting with e.g. the different  line $L': x=w, y=-z$  and intersecting $F$ with the pencil of planes $\Pi'_\mu: x-w=\mu(y+z)$, we obtain another pencil of elliptic curves $E'_\mu$, essentially fibers of the rational map $\mu={w-x\over y+z}:F \dashrightarrow\P_1$. The rational points on $F$ produce rational values of both maps $\lambda,\mu$ (when they are defined).\footnote{This double fibration by elliptic curves yields two sets  $\{[m]_i:m\in\Z\}, i=1,2$ of rational endomorphisms of $F$ obtained by integer multiplication on the two families; on using the compositions $[m]_1\circ [n]_2$, $m,n\in\Z$,  applied to a given rational point, we again obtain a Zariski-dense set of rational points.}

\bigskip

We shall now prove the HP for $F$, using both of the above  fibrations  by elliptic curves. We note that even with the help of the Vojta's conjecture an immediate deduction of the HP seems not obvious. Indeed, since $K_F=0$,  the canonical class of a cover $Y$ of $F$ is given by the ramification divisor $R$ (considered with multiplicities).
In case this is a `big' divisor, the Vojta's conjecture would imply that  $Y(k)$ is not Zariski-dense in $Y$ and then this cover would be negligible. But the bigness of $R$ is not guaranteed, and to take care of this  possibility appears implicitly also in our argument below.

\medskip

If the Hilbert Property fails for $F/\Q$,    then let $\pi_i:Y_i\to F$ be finitely many covers  each of degree $>1$ such that $F(\Q)-\bigcup \pi_i(Y_i(\Q))$  is not Zariski-dense. We may suppose that  the $Y_i,\pi_i$ are defined over $\Q$, for when this does not happen  the corresponding map produces a non-dense set of rational points on $F$, as is easy to see by conjugating over $\Q$.  

We shall inspect these covers by restricting them above the elliptic curves coming from the above fibrations. \footnote{We note that considering a single fibration (and merely the rational points coming from the above sections)  would not suffice in absence of additional information. Indeed, under e.g. the first fibration, $F$ may be seen as an elliptic curve over $\Q(\lambda)$, with a point of infinite order defined over $\Q(\lambda)$; then the group generated by this point would lift to the union of the two covers of this elliptic curve obtained by division by $2$ followed by suitable translations (as in the weak Mordell-Weil).  %As mentioned above, the double elliptic fibration for this type of surface is exploited for a different purpose also by Swinnerton-Dyer in \cite{SwD3}.
}

\medskip

Let $Y,\pi$ be one of these covers, where we can assume $Y$ to be projective and $\pi$ finite. Note that $F$ is smooth and simply connected, hence algebraically simply connected.  %this map may send finitely many curves into points of $F$, which would become  part of the branch locus (defined here as the image of the ramification locus). If the branch locus of $\pi$ consists only of (such) finitely many points, we have a contradiction, because $F(\C)$ deprived of a finite set is simply connected, and $X$ deprived of the fibers of those points would become a connected unramified cover of it, of degree $>1$.

Therefore the branch locus  of $\pi$ contains some irreducible curve $B$ on $F$.  %We also add that in the sequel we shall disregard the finitely many points arising as above (i.e. the points above which $\pi$ is not finite).

\medskip

Now,  %$B$ cannot simultaneously be of the shape $E_l$ and $E'_m$. Suppose e.g. that $B$ is not one of the $E_l$;
a first case occurs when $\lambda$ is not constant on $B$ (including the case  $B=L$), and  let us call   `of the first type' the covers containing such a $B$.

 This fact   % that  $B$ is not among the $E_l$
  implies that  $B$ shall generically meet the elliptic curves $E_l$. % (outside the mentioned finite set), so  a cover  of the first type restricted above a general $E_l$ shall be somewhere  ramified. 
\medskip

A subcase occurs when the cover is generically  reducible  above $E_l$ for $l\in\C$. Then %, since it is  irreducible over $F$,
 it becomes reducible after a base change to $C\times_{\P_1}F$, where $C$ is a suitable  curve (the fiber product being understood for the map $\lambda$ on $F$). But then, since the cover is irreducible over $F$,  the components  above $E_l$ may be defined over $\Q$ only for a thin set of rational values of $l$. That is, outside a thin set of $l\in\Q$, the cover restricted above $E_l$  is irreducible  over $\Q$ but reducible  over $\overline\Q$. It is a well known easy fact that this implies that the cover restricted to such a rational $l$ has only finitely many rational points (see e.g. \cite{SeTGT}, p. 20, first Remark). 

In conclusion, in this subcase we have only finitely many rational points above $E_l$, except possibly for a thin set of $l\in\Q$.

\medskip

The other subcase occurs when the cover is generically irreducible  over $E_l$, hence irreducible for all but finitely many $l\in\C$; now the (irreducible) curve above $E_l$ will have genus $>1$ and shall contain only finitely many rational points by Faltings' theorem. So we still have the previous conclusions, actually  for all but finitely many $l$. 

\medskip

We also conclude that if all the covers in question would be  of this type then for  `most' $E_l$ only finitely many rational points of $E_l$ would lift to the cover, and we would have a contradiction.

Let us then study separately the covers such that   the whole branch locus (apart from finitely many points) consists entirely of curves on which $\lambda$ is constant (so curves which are components of some  $E_l$); let us refer to these covers as being `of the second type'. 

Again we have two subcases, the first one being when the cover is generically reducible  above $E_l$. As before, apart from a thin set of rational $l$, only finitely many points may lift to the cover, and we may disregard these covers as well, and put them together with the ones of the first type.

We denote by $T$ the union of the (exceptional)  thin sets of rational numbers $l$  that we have just described; it is still a thin set (in $\Q$). 

\medskip

Suppose then that $Y$ is of the second type and  that it remains generically irreducible  above $E_l$ for $l\in\C$; we denote by $S$ the set of these covers. 

Then for a cover in $S$,  $\pi^{-1}(E_l)$ will be   irreducible for all but finitely many $l$ (and  we may assume these exceptional ones to be in $T$), unramified, and hence will become an elliptic curve $\cE_l$ (after a choice of origin), and the map $\pi|_{\cE_l}$ will be a translate of an isogeny.  For $l\in\Q-T$, our assumptions and the above arguments imply that all but finitely many rational points of $E_l$ lift to rational points of some $\cE_l$ (i.e. for a suitable $(Y,\pi)$ in $S$). 

Now, for all of these $l\in\Q$, the group $E_l(\Q)$ is a finitely generated abelian group of positive rank (by the above construction); and of course also the involved groups $\cE_l(\Q)$ are finitely generated, and  the $\pi(\cE_l(\Q))$ shall be translates of subgroups of $E_l(\Q)$. We now appeal to the following simple

\begin{lem} \label{L.all} Let $G$ be a finitely generated abelian group of positive rank,   let, for $u$ in a finite set $U$,  $H_u$ be  subgroups of $G$ and $h_u\in G$. Suppose that  $G-\bigcup_{u\in U}(h_u+H_u)$ is finite. Then this complement is actually empty.
\end{lem}

\begin{proof}[Proof of Lemma] Let $U'$ be the subset of $u\in U$ such that $H_u$ has the same rank $r>0$ of $G$. The intersection $H:=\bigcap_{u\in U'}H_u$ is clearly of rank $r$. If $\bar h_u$ denotes the image of $h_u$ in $G/H$ then $\bigcup_{u\in U'}(\bar h_u+(H_u/H))$ either covers $G/H$ or not.

 In the first case we have $\bigcup_{u\in U'}(h_u+H_u)=G$, proving the conclusion of the lemma (actually with a possibly smaller set of $u$ than needed). 
 
 In the second case, the complement of $\bigcup_{u\in U'}(h_u+H_u)$ in $G$ contains some coset of $H$, hence cannot be covered up to a finite set by finitely many cosets of subgroups of smaller rank, a contradiction which proves the lemma.
\end{proof}

Using the above remarks and the lemma we conclude that for $l\in\Q -T$   %$ outside a certain thin set) 
actually all the points in $E_l(\Q)$ are lifted to rational points on some cover in $S$. 

\medskip

Now, consider a curve $E'_m$, for some general enough $m\in\Q$, supposing also that $m$ has  large enough height so that $E'_m(\Q)$ is infinite. 

Since $E'_m$ is not among the $E_l$, it shall intersect each of the $E_l$  in a finite set, which shall be  nonempty for general $m$.  So, each cover in $S$ shall be somewhere ramified above $E'_m$. 

By the same argument as given above for the covers of the first type, and for $m$ outside a suitable  thin set  $T'$  of $\Q$, only finitely many points of $E'_m(\Q)$ can lift to rational points of covers in $S$.

However any rational point $p$  in $E'_m(\Q)$ lies also on some $E_l$, where $l$ is just the value $\lambda(p)$ of the map $\lambda$ at the rational point. 

Therefore, if we can find $m\in \Q -T'$ such that $\lambda(p)$ is not in $T$ for infinitely many $p\in E'_m(\Q)$, we have a contradiction with Lemma \ref{L.all}, which would say that all such points lift to some cover in $S$. 

\medskip

We can then suppose that, for any given  $m\in\Q-T'$ all but finitely many points  $p$ in  $E'_m(\Q)$ (which is an infinite set) are such that $\lambda(p)\in T$. Since $T$ is a thin set, it is a finite union of images 
$\varphi(Z(\Q))$, for morphisms $\varphi:Z\to\P_1$, of degree $>1$, from certain curves $Z$.  

Each such morphism will be branched over a nonempty finite subset  of $\P_1$ (actually containing at least two points) and let us denote   the union of these finite sets of branch points by $\cR\subset\P_1(\overline{\Q}))$. 

The set of branch points of $\lambda$ restricted to $E'_m$ depends {\it a priori} on $m$. Indeed, it turns out that each branch point depends on $m$, as we prove in the following

\begin{lem} For all $m$ such that $E'_m$ is smooth, the restriction to $E'_m$ of the rational function $\lambda$ is a  map of degree two, having four ramification points. The value  of $\lambda$ at each of these ramification points is a  non-constant (algebraic) function  of $m$.
\end{lem}

\begin{proof} This fact could be checked by a (nasty) computation, but we  avoid this. Let us fix   an $m$ so that $E'_m$ is a smooth cubic lying on the  plane $\Pi'_m$. In order to compute the degree of the map $\lambda$ restricted to $E'_m$, we have to compute the number of points in $E'_m$ where it takes a generic value. Now,   by our opening construction, a `value' of the map $\lambda$ is represented by a line passing through $p_m:=\Pi'_m\cap L=(1+m:1-m:1+m:1-m)\in E'_m$; so, since  such a line generically intersects $E'_m$ in two more points, the degree of $\lambda$ is $2$. The ramified values of $\lambda$ are the tangent lines passing through $p_m$, apart from  the tangent line at $p_m$ (unless $p_m$ is a flexus). 

We want to prove that these  values are not constant as $m$ varies. Of course, in order to compare  values of $\lambda$ for different values of $m$ we have to identify a value of $\lambda$ with the corresponding plane, not with a line on the $\mu$-plane $\Pi'_m$. More precisely, given the  plane $\Pi'_m$ and a point $p\in E'_m$,   $\lambda(p)$ will be the $\lambda$-plane (i.e. in the family $\Pi_\lambda$) generated by $p$ and $L$, i.e. the $\lambda$-plane generated by the line connecting $p$ with $p_m$ and the line $L$. It then makes sense to compare two   values  $\lambda(p), \lambda (q)$ even if $p$ and $q$ do not lie on a same $\mu$-plane.

 Now, if a ramified value for $\lambda$ were fixed, say equal to $l$, the  $\lambda$-plane $\Pi_l$ would contain  a tangent line to $E'_m$, passing through $p_m$, for all values of $m\in\P_1$. To see that this does not happen, consider the degenerate cases $E'_0$ and $E'_\infty$. These cubic curves degenerate into three lines, of which only one passes through $p_m$. For $E'_0$ this line has an equation $x-w=y-z=0$, while for $E'_\infty$ it is the line of equation $y+z=x+w=0$. These two lines are not coplanar, concluding the proof of the lemma. 
 
\end{proof}

We can now conclude the proof of the theorem. By this  lemma, the set of branch points of $\lambda$ restricted to $E'_m$ may intersect  $\cR$ at most for finitely many   $m\in\P_1$. Choose an  $m\in\Q-T'$ such that this does not happen Now fix a morphism $\varphi:Z\to\P_1$ as before. 
 The intersection $\varphi(Z(\Q))\cap \lambda(E'_m(\Q))$ is the image in $\P_1$ of the rational points in $W(\Q)$, where $W\to\P_1$ is obtained as a  fiber product of $\varphi: Z\to\P_1$ and $\lambda: E'_m\to\P_1$. Since $\varphi$ has degree $>1$,  $\varphi$ must ramify above some point of $\P_1$ (actually above at least   two points) and by our choice  these points are not branch points for $\lambda$. Hence the map $W\to E'_m$ must also ramify somewhere,  for each irreducible component of $W$, which implies that  each component of $W$ has genus $\geq 2$, so contains only finitely many rational points by Faltings' theorem.

 This proves that $\lambda(E'_m(\Q))\cap \varphi(Y(\Q))$ is finite for each morphism $\varphi$ of the given finite set, hence $\lambda(E'_m(\Q))\cap T$ is finite.
 
Summing up,  in all cases we obtain a contradiction, proving finally the sought result.

%Hence, since all points of $E_l(\Q)$ lift to rational points for some cover in $T$ we have a contradiction with the fact that $E'_p(\Q)$ is infinite.

\medskip

\begin{rem} (i)  Inspection shows that for the proof Faltings' theorem may be replaced by Mumford's earlier estimate for rational points on curves of genus $>1$.

(ii) The given argument applies for more general  K3 surfaces, in particular also for the intermediate K3 surface giving rise to the example % \ref{EX.HP},  % ??
treated in the next subsection (see Remark \ref{R.noHP}(iii) for a few words on this). We have preferred to stick to  the  special case of the Fermat surface for the sake of simplicity.

(iii) As already mentioned, we do not know whether this surface $F$ has the WAP or at least the WWAP; this is unlikely to hold if the only rational points in $F$ come from the above  construction. 
\end{rem} 

\subsection{An Enriques surface without the Hilbert Property and the proof of Theorem \ref{T.Enriques}}\label{SS.Enriques}  We shall now present the details of the example announced in the introduction, of a surface without the Hilbert Property, with a Zariski-dense set of rational points and admitting no non-constant maps to abelian varieties, hence proving   Theorem \ref{T.Enriques}.

This surface is a so-called {\it Enriques surface}, defined as a smooth surface with vanishing irregularity and geometric genus, and vanishing $2K$; no such surface can be rational, e.g. because of this last fact. Such surfaces can have a Zariski-dense set of rational points;   some explicit examples can be found e.g. in \cite{HS}. Actually, Bogomolov and Tschinkel \cite{BT} proved that after a finite extension of their field of definition, every Enriques surface  has a Zariski-dense set of rational points (potential densisty of rational points).

\medskip

To obtain our example, we start from the automorphism $\sigma$ of $F$ defined by
\begin{equation*}
\sigma(x:y:z:w)=(x:iy:-iz:-w).
\end{equation*}
It has order $4$ and no fixed points  in  $F$. Its square is $(x:y:z:w)\mapsto (x:-y:-z:w)$  and has the eight  fixed points $(1:0:0:i^m)$, $(0:1:i^m:0)$ for $0\le m<4$. Let $F'$ be the quotient of $F$ by the subgroup generated by $\sigma^2$ (see \cite{SeAGCF}, III.12, for quotients of a variety by a finite group). It turns out that a smooth model of $F'$ is also a K3 surface, on which $\sigma$ acts (as an automorphism of order $2$) without fixed points. In fact, a fixed point of $\sigma$ on $F'$ would correspond to an orbit $p,\sigma^2(p)$ on $F$, fixed by $\sigma$; this would imply either $\sigma(p)=p$ or $\sigma(p)=\sigma^2(p)$, so either $p$ or $\sigma(p)$ would be fixed by $\sigma$ on $F$, which is not the case. So $F'$ is an unramified
cover of its quotient by the subgroup of order $2$ generated by $\sigma$. A smooth model of this quotient is the desired Enriques surface.

We omit a discussion of smooth models and also on the quotient varieties so obtained  (which is not in fact necessary for our assertions) and only say that  a surface birational to the Enriques, call it  is $\cE$,  may be defined in $\P_3$ by the equation appearing in Theorem \ref{T.Enriques}, i.e. 
\begin{equation*}
\cE:  x_0x_2^4+x_1x_3^4=x_0^2x_1^3+x_0^3x_1^2,
\end{equation*}
(with singular points  $(1:0:0:0)$ and $(0:1:0:0)$) whereas $F'$ is birational to the cover $\cF$ of $\cE$ given in $\P_4$ by the further equation $x_0x_1=x_4^2$. % (We again mention that this surface is simply connected, whereas this is not the case both for a smooth model of it and also for  the quotient variety described above.)

Note the dominant rational map from $F$ to $\cE$ defined by 
\begin{equation*}
x_0=x^4, x_1=y^4, x_2=xy^2z, x_3=x^2yw,
\end{equation*}  
where the four expressions are invariant by $\sigma$. Through the rational points that we have observed on $F$, this yields a Zariski-dense set of rational points on $\cE$.

\medskip

%%%%  PROMEMORIA:  Le curve ellittiche della prima famiglia sulla Fermat diventano le curve $x_2=\kappa x_1(x_1-x_3)^2 sulla intermedia data da x_2^4=x_1^2(x_1^2+x_1^4-x_3^4) in forma non omogenea. Si trova eliminando l'equazione \kappa^4 x_1^2(x_1-x_3)^8=x_1^2+x_1^4-x_3^2 e con x_3=zx_1 si trova ...   ?? dovrebbero venire curve al pi\9D' ellittiche..

We continue by proving Theorem \ref{T.Enriques}, where we could use Theorem \ref{T.sc-HP}, through  the above variant  of the Chevalley-Weil Theorem; but it is easier to argue directly,  on showing  that any rational point on $\cE$ lifts  either to $\cF$ or to the similar cover $\cF^-$ defined by the equation $x_0x_1=-x_4^2$. 

But before this, let us observe, independently of the above interpretation of $\cE,\cF$ as related to quotients of $F$, that  these covers are indeed irreducible. For this, we may use the functions $u_i:=x_i/x_0$ and the equation $u_2^4=u_1(u_1^2+u_1-u_3^4)$ for $\cE$, and the further equation $u_4^2=\pm u_1$ for $\cF, \cF^-$ resp. (exhibiting our surfaces as covers of the $(u_1,u_3)$-plane).   Since none of the factors on the right side of the first equation is a square in $\C(u_1,u_3)$, it follows that the equations yield linearly disjoint extensions of $\C(u_1,u_3)$, proving what we need.

% Indeed, the polynomial $t^6+t^4-t^2u_3^4-u_2^4$ is irreducible in $\C(u_2,u_3)[t]$, and it follows that   $\cE,\cF,\cF^-$ are indeed all irreducible, and $\cF, \cF^-$ are covers of $\cE$ of degree $2$ without rational sections.

\medskip

Let  now $(a_0:a_1:a_2:a_3)\in\cE(\Q)$, where $a_i$ are coprime integers. It suffices to show that one among $\pm a_0a_1$ is a square, where we can assume that $a_0a_1\neq 0$. For this, let $p$ be a prime number and let $p^{e_i}||a_i$, so $e_i$ are integers $\ge 0$. If  the $p$-adic order of $a_0a_1$ is not even,  then $e_0+e_1$ is odd. But then the numbers $e_0+4e_2, e_1+4e_3, 2e_0+3e_1, 3e_0+2e_1$ are pairwise distinct, because they are congruent modulo $4$ resp.  to $e_0,e_1,e_1+2,e_0+2$, and hence actually   pairwise   incongruent modulo $4$. However this is impossible because  these numbers are the $p$-adic orders of the four terms appearing in  the equation defining $\cE$. 

Hence $a_0a_1$ has even $p$-adic order  at every prime, proving our contention.\footnote{The argument yields directly that $\cE(\Q)$ lifts to $\cF(\Q(i))$, which is analogue to the standard Chevalley-Weil theorem. The present device, on replacing one cover by two covers to maintain the ground field, is similar to the above version of   Chevalley-Weil.}    

Of course, this shows that $\cE$ has not the Hilbert Property, as stated. 

\begin{rem} \label{R.noHP}  (i) Exactly the same proof shows that for every discrete valuation $\nu$ of a field $K$, and for each point $P\in\cE(K)$, the  valuation of   $(x_1/x_0)(P)$   is even.  If we apply this with $K$ equal to   the function field of $\cE$, and observe that $x_1/x_0$ is not a square in that function field, we obtain  that  every {\it smooth} model of $\cE$ admits an irreducible unramified double cover, namely the one corresponding to taking the square root of $x_1/x_0$. Also the quotient variety described above (i.e. $F'/<\sigma>$) is not simply connected, because $\sigma$ has no fixed points on $F'$. 
On the contrary, as already said, this singular model is (topologically, so algebraically) simply connected.\footnote{We have preferrred to use this model of a hypersurface of $\P_3$ both for simplicity and also because it illustrates some subtleties related to the fundamental group of different models.} 

(ii) The argument shows easily that $\cE$ has not the Hilbert Property over any number field $k$: again, we can use the refined Chevalley-Weil Theorem (applying in practice Theorem \ref{T.sc-HP}) or argue directly, using that elements having even valuation at all primes are squares up to finitely many factors (depending on $k$). 

(iii) It is possible to prove the HP for the above surface  $F'=F/\sigma^2$; we outline the argument. Using the HP for $F$ (i.e. Theorem \ref{T.Fermat}) the crucial issue is to produce rational points on $F'$ not coming from covers isomorphic to $F$ (over $\overline\Q$). In turn this amounts to find `sufficiently many'  points on $F$ defined over a quadratic field, on which $\sigma^2$ acts as a conjugation. Our formulae translate this problem into rational points for   twisted models of $F$ defined by $x^2+d^2y^4=d^2z^4+w^4$, where $d$ is a nonzero integer, not a square.  Given then finitely many covers of $F'$ of degree $>1$, it is a matter of routine  to pick a suitable definite $d$ (depending on the covers) and repeat the above proof pattern of Theorem \ref{T.Fermat} for   this new surface, concluding the argument. 
\end{rem}

\subsection{A surface with the Hilbert Property but without the Weak Approximation Property and the proof of Theorem \ref{T.HP-WAP}}\label{SS.HW}
In this subsection we are going to exhibit an example of a surface $X/\Q$ with the HP but without the WAP, proving in particular Theorem \ref{T.HP-WAP}. 

The negation of WAP is  due in fact  to Swinnerton-Dyer \cite{SwD2} (who used the real valuation).  We define $X$ as a cubic  in $\P_3$, with coordinates $(t:x:y:z)$, by 
\begin{equation}\label{eq.chatelet1}
X: t(x^2+y^2)=(4z-7t)(z^2-2t^2).
\end{equation}\label{chatelet1}

It is singular precisely at the points $(0:1:\pm i:0)$. We observe at once that $X$ becomes rational over $\Q(i)$ (where the quadratic form $x^2+y^2$ is equivalent to $xy$). In particular, it is simply connected and has the WAP and  the  HP over $\Q(i)$. 

Actually, as stated in Theorem \ref{T.HP-WAP}, $X$ has the HP even over $\Q$. To see this, proving then half of the theorem, we argue somewhat similarly to the case of the Fermat surface $F$ ot Theorem \ref{T.Fermat}. Again, though this is simpler than before, we do not see any direct deduction of the HP from general known results.

 %: observing that $X$ contains the line $L:t=z=0$, we intersect it with the pencil of planes through this line, say the  $\Pi_\lambda$ with equation  $z=\lambda t$.  The plane  $\Pi_l$ intersects $X$ in the said line plus the conic on the plane with further equation $ x^2+y^2=(4l-7)(l^2-2)$. 
Confining to the affine subset $t\neq 0$, we divide by $t$ and, setting $x/t=\xi, y/t=\mu, z/t=\lambda$, we find a    family of  conics in coordinates $\xi,\mu$:
\begin{equation*}
C_\lambda : \xi^2+\mu^2=(4\lambda-7)(\lambda^2-2).
\end{equation*}
and a family of
 elliptic curves  in coordinates $\lambda, \xi$:
\begin{equation*}
E_\mu : \xi^2=(4\lambda-7)(\lambda^2-2)-\mu^2. \qquad  \footnote{For given $\mu$, we may see this as a curve in the plane $(\lambda,\xi)$ and take as origin  the point at infinity; in our setting this  corresponds to the point $(0:0:1:0)$ on $X$, though this is immaterial for us.} 
\end{equation*}

Now, for instance, we find the (smooth)  rational point $x=y=t=1, z=2$ on $X$. As remarked in \cite{SwD2}, one may obtain infinitely many rational points by intersecting $X$ with the tangent plane at the  rational point, which would even show that $X$ is unirational over $\Q$; however here we shall proceed differently. Before this, we pause for a remark:

\begin{rem} 
%As mentioned above,   Colliot-Th\'el\`ene has conjectured that any unirational variety has the WWAP, and therefore the HP;  however this has not been proved, and seems quite difficult (as it would imply a positive answer to the inverse Galois problem). In the Appendix below we shall prove a form of WWAP for this surface.

Certainly the rational points coming from a single  unirational parametrisation would not be sufficient to confirm the HP: in fact, since this surface is not rational (over $\Q$), these points would all come, by the very definition,  from a cover of degree $>1$, namely the rational cover which exhibits unirationality. Hence a proof of the HP requires to produce other rational points.
   \end{rem} % \medskip

To go ahead, observe that the said rational point yields the rational point $\xi=1, \lambda=2$ on the elliptic curve $E_1$  %with Weierstrass equation $x^2=(4l-7)(l^2-2)-1=4l^3-7l^2-8l+13$, 
 which  may be checked to be non torsion. 

Then we find infinitely many rational points $(\xi_n,\lambda_n)\in E_1(\Q)$ and rational points $(1:\xi_n:1:\lambda_n)\in X(\Q)$. Now,  this  also yields  that the conic $C_{\lambda_n}$ %affine circle $x^2+y^2=(4l_n-7)(l_n^2-2)$ 
has the rational point $(\xi_n,1)$ and thus infinitely many rational points, proving that $X(\Q)$ is Zariski-dense in $X$.

\medskip

To prove more, i.e. the HP, we proceed similarly to the case of the Fermat surface (however things shall be simpler now). 

Base change by the map $\mu: C_2\to\P_1$ yields a section $s$ to the surface: i.e. for $p=(a,b)\in C_2$  we set $s(p)=(2,a)\in E_b$. This is not identically torsion, because e.g. the map $\mu$ is branched above $\mu=\pm\sqrt 2$, whereas the bad reduction of the family $E_\mu$ occurs on a linearly disjoint field. Therefore a corollary of Silverman's theorem implies that $s(p)$ may be torsion on $E_b$ only for finitely many rational points $p=(a,b)\in C_2(\Q)$. In particular, for all but finitely many such rational values $b$, we have that $E_b(\Q)$ is infinite.

Note that for rational $b$ a rational point $q=(u,v)\in E_b(\Q)$ produces a rational point $(v,b)$ on $C_u$, and therefore $C_u$ becomes birational to $\P_1$ over $\Q$ and has `many' rational points. 

For later reference, let us denote by $A$ the set of rational numbers $l$ such that $C_l$ has a rational point; the given argument says that 

\medskip

\centerline{\it $l\in A$ whenever $l=\lambda(q)$ for a $q\in E_b(\Q)$, some $b\in\mu(C_2(\Q))$.}

\medskip

 Suppose now that $X$ has not the Hilbert Property (over $\Q$), and let  $\pi: Y\to X$ be one of the covers of $X$ which occur in this violation of HP, of degree $>1$, defined over $\Q$. We distinguish between two types of such covers.
 
 - {\tt First type}: this occurs when $Y$ remains generically irreducible above $C_\lambda$ (for $\lambda\in\C$). 
 
 Then $Y$ may be reducible above  $C_l$ only for finitely many $l\in \C$.  For each of the remaining $l\in A$, the image $\pi(Y(\Q))$ by definition covers at most a thin set in $C_l(\Q)$.
 
 - {\tt Second type}: this occurs when $Y$ is generically reducible above  $C_\lambda$ (i.e. reducible for generic  values of $\lambda$ in $\C$).  Now, for a special value $l\in\Q$,  either $Y$ is irreducible over $\Q$  above $C_l$ (but it is certainly reducible over $\overline\Q$ since it is of second type) or not. In the first sub-case, being reducible over $\overline\Q$ but irreducible over $\Q$, it may have only finitely many rational points (above $C_l$).  The second sub-case   may  happen (taking into account simultaneously all finitely many covers  of this second type) only for a thin set $T$ of values $l\in\Q$, i.e. the set of these rational $l$ is covered by  images $\varphi(Z)$, for finitely many morphisms 
 $\varphi:Z\to \P_1$ of degree $>1$. 
 
 \medskip
 
 The danger now is that $T$ contains all the $l\in\Q$ such that $C_l$ has a rational point (i.e. that $T$ contains $A$).
In particular, in view of the above remarks  in this `bad' case $T$ would contain all values of $\lambda$ on the rational points of a curve $E_b$, where $b$ can be any element of $\mu(C_2(\Q))$. 

To exclude this, we note that the set of branch points of $\lambda$ on $E_u$ consists of $\infty$ and the roots of the polynomial  $(4\lambda-7)(\lambda^2-2)-u^2$. Now, each given complex number $c$ can be such a root only for finitely many $u\in\C$. But then we can find a $b$ in the infinite set $\mu(C_2(\Q))$ such that the set of branch points of $\lambda:E_b\to \P_1$ intersects the set of branch points of $\varphi: Z\to\P_1$ at most at infinity, and this for all of the $\varphi$ in question.

Under this condition, (each component of)  the pullback cover $\lambda^*(Z)$ of $E_b$ must be ramified, hence of genus $>1$ and can have only finitely many rational points. Therefore for such numbers $b\in\Q$ we find that $\lambda(E_b(\Q))\cap T$ is finite, whence we may pick an $l\in \lambda(E_b(\Q))- T$. For such an $l$ the covers of the second type as well absorb at most a thin subset of the rational points of $C_l$, proving finally what we need.

\bigskip

It remains to check that $X$ has not the WAP. Now, this is done in \cite{SwD2}, where it is actually proved that there are no rational points in the real component of $X(\R)$ defined by $|z/t|\le \sqrt 2$. This itself excludes the WAP  already for the real place of $\Q$. In the next Remark we shall give an argument similar to the one in \cite{SwD2}, for excluding the WAP for a similar surface, for which the rational points are dense in the real ones.

\begin{rem}
 (i)  The slightly different  surface given by the equation
 \begin{equation}\label{eq.chatelet2}
 t(x^2+y^2)=(2z-7t)(z^2-2t^2)
 \end{equation}
 has rational points in both connected components of the real points (e.g. $(1:2:1:1)$ in the component $|z/t|\le \sqrt 2$ and $(1:13:1:6)$ in the other one).  Then it  is  not too difficult to prove that the rational points are dense in the real points. However again the WAP fails:  for instance there are no rational points  $(a:b:c:d)$ in the second component, with coprime integer coordinates  and such that $a>0$, $a\equiv d\equiv 1\pmod 4$, $b\equiv 0,c\equiv 1\pmod 2$ (conditions which define  a certain $2$-adic neighbourhood of $(1:2:1:1)$).  For otherwise   we would have $2d-7a\equiv -1\pmod 4$ and $2d-7a$ would be a positive integer dividing $a(b^2+c^2)$. Let then $p$ be a prime $\equiv 3\pmod 4$ and dividing $2d-7a$ to an odd power $h$ (which would exist). Now, $p$ divides $b^2+c^2$ to an even power, which is positive if and only if it divides both $b,c$.  It follows that if $p$ would not divide $a$ then $p$  would have to divide $d^2-2a^2$ as well, hence also $41a^2$, a contradiction. Therefore $p^l|| a$, where $l>0$, and hence $p$ divides $d$, and cannot divide both $b,c$, hence it does not divide $b^2+c^2$.  We conclude that $l=h+m$, where $p^m||d^2-2a^2$ and it easily follows that $p^h||d$ and $m=2h, l=3h$. But then, repeating the argument for all such $p$ we would conclude that $a\equiv 2d-7a\pmod 4$, a contradiction.
 
 (ii)  These last surfaces are unirational, and thus expected to have the WWAP (according to a conjecture of Colliot-Th\'el\`ene mentioned above).  Actually equations \eqref{eq.chatelet1} and \eqref{eq.chatelet2} define   so-called Ch\^{a}telet surfaces; for these surfaces, WWAP was proved by Colliot-Th\'el\`ene, Sansuc and Swinnerton-Dyer in \cite{CT-S-SD}; the case of cubic Ch\^{a}telet surface appeared also in \cite{SwD4}. 
 For completeness, in the Appendix 2 below we sketch a proof, for the surface of the theorem,  of a property very similar (albeit weaker) than the WWAP; it seems that  those  arguments should give the full WWAP.
 %We cannot prove this, though it should possible to show that, for each large prime $p$, there are rational points in each compatible congruence class mod $p$. For instance, tale  a unirational map of two parameters $\lambda,\mu$ and suppose that  $z(\lambda,\mu)-z_0$ is  absolutely irreducible modulo $p$  for each  $z_0\in \F_p$. Then the curve $z(\lambda,\mu)=z_0$ would have some rational point over $\F_p$. Now, using the parameterzation of the circle, one should be able to complete the argument.
 
 In any case, we remark again that we do not know whether the WWAP is  generally implied by the HP.

 \end{rem}

 %%%%%%%%%%%%%%%%%%%%%%%%%%%%%%%%%%%%%%%%%%%%%%%%%%%%%%%%%%%%%%%%%%%%%%%%%%%%%%%%%%%%%%%%%%%%%%%%%%%%%%%%%%%%%%%%%%%% APPENDIX   %%%%%%%%%%%%%%%%%%%%%%%%%%%%%%%%%%%%%%%%%%%%%%%%%%%%%%%%%%%%%%%%%%%%%%

 \section{Appendices}

\subsection{\tt Appendix 1: On the Hilbert Property for Kummer surfaces } 

In this Appendix we consider Kummer surfaces over $\Q$, i.e. quotients $X=A/\pm I$ of an abelian surface $A$ defined over $\Q$; we assume that $A$ (and thus $X$) has  a Zariski-dense set of rational points. 

We shall sketch a proof  that 

\centerline{\it There is a Zariski-dense set of  rational points of $X$ which may not be lifted to $A$,}

 \noindent providing a piece of the Hilbert property for $X$. We remark that the full Hilbert property for $X$ would follow from this fact combined with the Vojta's conjectures; this second part could  presumably be dealt with unconditionally by the methods of \cite{Z}, where we hope that someone (we ?) will  undertake the task to fill the details in. \footnote{These reductions rely on the fact that  a smooth model of $X$ is known to be simply connected, and all its - ramified - covers not factoring via the canonical map $A\to X$ come  from ramified covers of $A$.}
 
\medskip

The statement is invariant by isogeny, and thus we may split the proof into  the two basic cases  when $A$ is  either the product of two elliptic curves or  $A$ is simple; in turn, in this second case $A$ is known to be isogenous to the Jacobian of a curve of genus $2$, where we limit here ourselves to the case when everything is defined over $\Q$. We assume that $A$ has a Zariski-dense set of rational points.

\smallskip

{\it First case}: $A=E_1\times E_2$. Here we suppose that the elliptic curve $E_i$ is given by Weierstrass equation $y^2=f_i(x)$, for cubic polynomials $f_1,f_2$ without multiple roots. Then $X$ is given birationally by the equation 
\begin{equation*}
f_2(x_2)=w^2 f_1(x_1).
\end{equation*}
Then, the above statement on   points in $X(\Q)$ which do not come from $A(\Q)$ translates into  the following assertion:

\medskip

{\it If both curves $y^2=f_i(x)$ have infinitely many rational points, there exists a Zariski-dense set of rational solutions of the displayed  equation such that $f_2(x_2)$ is not a square.}

\medskip

For this we consider the 2-dimensional family  $Z$ of curves $Z_{u_1,u_2}:f_1(u_1)f_2(x_2)=f_2(u_2)f_1(x_1)$ parametrized by $(u_1,u_2)\in\A^2$;  we have given an affine equation, but we implicitly consider the projective closure with respect to the plane $(x_1,x_2)$. 

The generic member has genus $1$ and there is a section  given by $(x_1,x_2)=(u_1,u_2)$. This  section can be taken as origin, giving $Z$ the structure of an elliptic curve $\cE$ defined over $\Q(u_1,u_2)$.

We can obtain a second section by the `tangent method', i.e. intersecting the cubic $Z_{u_1,u_2}$ with the tangent line  at $(u_1,u_2)$. Now, the difference of the sections gives a point on $\cE$, defined over $\Q(u_1,u_2)$. It may be checked that this point is not identically torsion.\footnote{A priori the torsion order would be absolutely bounded, because a rational torsion point over $\Q(u_1,u_2)$ cannot have `too large' order, due to the Galois theory of Fricke and Weber, no need to use the deeper arithmetical results here. But we do not even need to use this, or to perform computations: indeed, $Z_{u_1,u_2}$  depends only on $w:=f_2(u_2)/f_1(u_1)$, whereas we may  vary both  $u_1,u_2$; %also,   the torsion order would have to be constant, %which would give a closed condition, which may be seen to be not possible.
then  we would obtain a continuous family of points of finite order on a same elliptic curve, which is impossible.} 

Then by the results of Silverman, specialising first $u_1=a_1$, for a rational point $(a_1,b_1)\in E_1(\Q)$, then $u_2=a_2$, for a rational point  $(a_2,b_2)\in E_2(\Q)$,  of large enough height in both cases, we obtain a non torsion rational point  on the elliptic curve $b_1^2f_2(x_2)=b_2^2f_1(x_1)$.  Multiples of this point  are dense of the curve $Z_{a_1,a_2}$ and on varying $a_1,a_2$ this clearly yields the required density of rational points on our surface.

\medskip

{\it Second case: $A=$ the Jacobian of a curve $H$ of genus $2$ defined over $\Q$}.  We let $y^2=f(x)$ be an equation for $H$, where we may take $f\in \Q[x]$ to be a polynomial of degree $5$ without multiple roots. We may suppose $A$ to be simple, for otherwise we fall  in the previous case after isogeny. 

There is one point $\infty$ at infinity on a smooth model of $H$ (i.e. the unique pole of $x$) and we may  embed  $H$ in $A$  on defining, for $p\in H$, $[p]=$class of the divisor $(p)-(\infty)$. Then $A$ is the set of sums $[p]+[q]$, for $p,q\in H$ (where the sums $[p]+[p']$ yield $0$, $p\mapsto p'$ being the canonical involution on $H$, because $p,p'$ are the zeros of the function  $x-x(p)$).

 We have an infinity of quadratic points on $H$ obtained by putting $x=a\in\Q$ and taking $y$ as a square root of $f(a)$ (omitting the cases when $f(a)$ is a square, a thin and actually finite set).
 
  For $p$ one such quadratic point observe that $\phi(p):=2[p]\in A$ is a point defined over a quadratic extension of $\Q$ and satisfies $\phi(p)^\sigma+\phi(p)=0$, where $\sigma$ is the nontrivial automorphism of the quadratic field $\Q(\sqrt{f(a)})$; indeed. $\phi(p)^\sigma=\phi(p^\sigma)=\phi(p')=-\phi(p)$, because of the above observation $[p]+[p']=0$.  We have already observed in \S \ref{SS.qc} that these points yield rational points on $X$ which do not come from rational points on $A$.

  To find a Zariski-dense set of such points, observe that  $n\phi(p)$, for $n\in\Z$ denoting  multiplication  in $J$, are points with  the same property  as $\phi(p)$ (relative to the same quadratic field). Certainly the points $\phi(p)$ cannot  be all torsion on $A$, and taking their images in $X$ yields rational points of $X$, and provides the required dense set.
 
 \begin{rem} We observe that the multiples $n\phi(p)$, when written in the shape $[u]+[v]$, for $u,v\in H$, correspond generally to points such that $u^\sigma=v'$ and such that, putting $u=(a,b)$, we have $a,b$ defined over our quadratic field but  such that $a\not\in\Q$. 
 
 Also, we remark that the same methods may be applied over any number field. 
  \end{rem}

\medskip 
 
 \subsection{\tt Appendix 2: WWAP for a cubic surface}
 In this Appendix we shall sketch a proof of a form of the WWAP for the surface $X$ of \S \ref{SS.HW}; for a proof of a more general result see \cite{CT-S-SD}. Namely, we shall consider only the following property: {\it for a large enough $l_0$, for every finite set $S$ of primes $\ell >l_0$, and every set $\xi_\ell\in X(\F_\ell)$, $\ell\in S$, there exists a rational point $\x\in X(\Q)$ such that $\x\equiv \xi_\ell\pmod \ell$ for all $\ell\in S$.}
 
 We note that this property, {\it a priori } weaker than WWAP, is sufficient to imply the HP (by the same arguments given e.g. in \cite{SeTGT}). And actually even less would suffice;  for instance we could restrict the $\xi_\ell$ to lie outside a  prescribed proper subvariety of $X$. In fact, we shall give detail only for this last  even weaker property, and only indicate how one can remove this restriction. %(Complete detail shall be given in possible future work.)
 
 In any case, in particular, the argument shall give as a byproduct an independent proof that $X/\Q$ has the HP.
 
 \medskip
 
 Before going ahead, we let $k$ be any  ground field of definition for $X$ (possibly char $k>0$), and recall the above notation, in particular the rational maps $\lambda,\mu: X\dashrightarrow \P_1$  which fiber $X$ by conics $C_\lambda$ and elliptic curves $E_\mu$. 
 
 The affine conic  $C:x^2+y^2=1$, isomorphic to the algebraic group $SO_2$, acts simply transitively on $C_\lambda$ as $(\xi,\mu)\mapsto (\xi x-\mu y,\xi y+\mu x)$, which also extends to yield a group of automorphisms of $X$. For $g\in C$, we shall denote by $q\mapsto g(q)$ this action, for $q\in X$.
 
 Corresponding to the elliptic curves $E_\mu$, we have multiplication maps $[m]:E_\mu\to E_\mu$, and we shall denote again by $[m]$ the rational map $[m]:X\dashrightarrow X$ obtained on viewing $X$ as an elliptic curve over $k(\mu)$. 
 
 By combining these maps, we can obtain a lot of rational self-maps on $X$ which are dominant and defined over the prime subfield of $k$. In particular, we can produce rational points starting from a given one. 
 For the application we are discussing, it turns out that it suffices to use the dominant rational maps (of degree $4$) 
 \begin{equation*}
 \sigma_g:=[2]\circ g:X\dashrightarrow X,\qquad g\in C.
 \end{equation*}
 We note that composing maps from  these families, and looking at the orbits of a given point, we can exhibit $X$ as a unirational variety over the prime field.
 
  We denote by $X_0$ the affine open subset $t\neq 0$ on $X$, with coordinates $\xi,\mu,\lambda$ and start from the point $p_0=(1,1,2)\in X_0$.  By acting with $\sigma_g$ we obtain the following points:  if $g^*:=g(p_0)=(a,b)\in C_2$, $b\neq 0$, then
  \begin{equation*}
 \sigma_g(p_0)=(a,s,r), \quad r=9({1\over b^2}-{1\over 4}),\quad s=-{6\over b}(r-2)-b.
 \end{equation*}
 Now, if $h\in C$, we obtain points $h(\sigma_g(p_0))=(h(a,s),r)$. Putting $h(a,s)=:(u,v)$, supposing $v\neq 0$  and duplicating, we obtain, on denoting  $f(z)=(4z-7)(z^2-2)$,
\begin{equation*}
\lambda( \sigma_h\sigma_g(p_0))={7\over 4}+{f'(r)^2\over 16 v^2}-2r.
\end{equation*}

Let now as above $k$ be a ground field, %, assumed to be algebraically closed (which shall be either $\bar\Q$ or an algebraic closure of $\F_\ell$), 
and let us choose $z_0\in k$; we try to equate the last displayed value to $z_0$ (with $v\neq 0$), which amounts to 
\begin{equation}\label{E.A}
4v^2(4z_0+8r-7)=f'(r)^2.
\end{equation}
We view this equation as defining (for given $z_0$) a curve $\Gamma$  in $C_2\times C$, with coordinates $g^*\times h=(a,b)\times h=: (a,b)\times (\alpha,\beta)$.  Note for future reference that in these coordinates we have $v=a\beta+ s\alpha$, whereas $r$ is given above in terms of $b$.

\medskip

Suppose for the moment that $z_0\in\F_\ell$ and that  $\Gamma$  is absolutely irreducible (over $\F_\ell$). Then, this curve would have absolutely bounded genus, so for $\ell$ large enough  Weil's Riemann-hypothesis for curves over finite fields would imply the existence of (several)  points   $g_0^*\times h_0\in \Gamma(\F_\ell)$, i.e. points in $(C_2\times C)(\F_\ell)$ solving the equation \eqref{E.A}, and we may also require that  $v\neq 0$. 

But now, since $C_2,C$ are rational and smooth over $\Q$, and thus have  the WAP, we may lift $g_0^*,h_0$ to points in $C(\Q)$, integral at $\ell$;  composing then the corresponding maps $\sigma$, starting from $p_0$,    we would  obtain a rational point $\w:=(\xi_0,\mu_0,\lambda_0)\in X_0(\Q)$, integral at $\ell$ and such that the reduction modulo $\ell$ of $\lambda_0:=\lambda(\w)$ is defined and equal to $z_0$. 

Finally, take  a point $\tilde \x\in X(\F_\ell)$ with $\lambda(\tilde\x)=z_0$. Now, using once more the transitive action  of $C$ on   the conic $C_{\lambda_0}$ and the WAP for $C$,  we can deform $\w$ to a rational point  on $C_{\lambda_0}$ with reduction $\tilde \x$ modulo $\ell$.

\medskip

We observe that use of the WAP for $C$ (in practice the Chinese theorem) allows to deal simultaneously with any finite set of primes $\ell$, provided they are large enough.

\medskip

To conclude we have then to investigate the possible values $z_0\in k$ making  the above curve $\Gamma$ reducible over $\bar k$ (where for the present purpose we take $k=\F_\ell$).  For this task, in the above formulae we take $g^*,h$ as generic points of $C_2/k, C/k$ related by the equation \eqref{E.A}.

We have a rational map $\psi:C_2\times C\to \P_1^2$, $(a,b)\times (\alpha,\beta)\mapsto (r,v)$, where $r$ is given by the above formula and $v=a\beta+s\alpha$.  It is easy to check that $\psi$ has degree $8$. Indeed, $[k(C_2):k(b^2)]=4$ and $v$ has degree $2$ as a map on $C$, for fixed $a,b$. 

Denoting $w:=b^2$, the extension $k(C_2)/k(r)$ is generated by $\sqrt w, \sqrt{2-w}$, whereas $k(C_2\times C)=k(C_2)(v,\alpha)$. Now,  a minimal equation for $\alpha$ is $(a\alpha)^2+(v-s\alpha)^2=a^2$, i.e. $(a^2+s^2)\alpha^2-2vs\alpha+v^2-a^2=0$, with discriminant $4(v^2s^2-(v^2-a^2)(a^2+s^2))=4a^2(a^2+s^2-v^2)=4a^2(f(r)-v^2)$. 

Then we find that $k(C_2\times C)=k(r,v)(\sqrt w,\sqrt{2-w},\sqrt{f(r)-v^2})$. 

Now, the equation  \eqref{E.A}  defines an irreducible curve $\Delta$ in the $(r,v)$-plane, and we wish to check whether its lift  $\Gamma$ to $C_2\times C$ remains irreducible.

 If 
$\Lambda:=w(4z_0+8r-7)=(4z_0-25)w+72$, then the equation for $\Gamma$ reads $4v^2\Lambda=wf'(r)^2$. Hence above $\Delta$ the curve may split only if $\Lambda$ becomes a square in the field $k(\sqrt w,\sqrt{2-w},\sqrt{f(r)-v^2})$, where $v$ is obtained using the equation for $\Delta$, i.e. the last square root is to be replaced by $\sqrt{(4\Lambda f(r)-wf'(r)^2)/\Lambda}$. Note also that $r=9(w^{-1}-4^{-1})$. 

 This may happen only if there is a  multiplicative  combination of $\Lambda, w, 2-w, (4\Lambda f(r)-wf'(r)^2)/\Lambda$ equal to a square in $k(w)$, with odd exponent of $\Lambda$.  In turn, this implies that either $\Lambda$ or $w^3(4\Lambda f(r)-wf'(r)^2)$ has a root $w=0$ or $w=2$, or is constant. The first cases yield $4z_0=-11$, whereas $\Lambda$ constant occurs only for $4z_0=25$.  As to the last mentioned case, this  is easily checked to be impossible. 

\medskip

So, finally we have proved the following  

\begin{thm}\label{T.A}  There is a (computable) number $l_0$ such that if $\cL$ is a finite set of primes $\ell >l_0$ and $\xi_\ell$ are points in the open subset  of $X_0$ defined by $\lambda\neq 25/4, -11/4$, then there exists a rational point $\x\in X_0(\Q)$ such that $\x\equiv \xi_\ell\pmod\ell$ for all $\ell\in\cL$.
\end{thm}

\begin{rem}  To get rid of the restrictions on the $\lambda$-coordinate (i.e. the above $z_0$), several methods are available.  For instance we could start  with  a different $p_0$, or use further maps similar to the $\sigma_g$ (possibly  obtained by multiplication by integers $>2$).  %Or we can also start by approximating to a suitable multiple of the point (on the relevant elliptic curve) such that the new $\lambda$-coordinate does not fall again into one of the three bad cases. Then, if the multiple is coprime to the order of the point modulo $\ell$, we can get back to the point by taking eventually another suitable multiple. 

It seems also quite feasible to carry out a similar proof for approximations in $X(\Q_\ell)$ with arbitrary accuracy, instead of merely modulo $\ell$, then achieving the full WWAP.  Since this would not add any further idea, and since the result is proved in a different way (and greater generality)  in \cite{CT-S-SD}, we do not pursue this task in the present paper. 
\end{rem}

%%%%%%%%%%%%%%%%%%%%%%%%%%%%%%%%%%%%%%%%%%%%%%%%%%%%%%%%%%%%%%%%%%%%%%%%%%%%%%%%%%%%%%%%%%%%%%%%%%%%%%%%%%%%%%%%%%%%%%%%%%%%%%%%%%%%%%%%  EXAMPLE %%%%%%%%%%%%%%%%%%%%%%%%%%%%%%%%%%%%%%%%%%%%%%%%%%%%%%%%%%%%%%%%%%%%%%

%%%%%%%%%%%%%%%%%%%%%%%%%%%%%%%%%%%%%%%%%%%%%%%%%%%%%%%%%%%%%%%%%%%%%%%%%%%%%%%%%%%%%%%%%%%%%%%%%%%%%%%%%%%%%%%%%%%%%%%%%%%%%%%%%%%%%%%%  LETTER %%%%%%%%%%%%%%%%%%%%%%%%%%%%%%%%%%%%%%%%%%%%%%%%%%%%%%%%%%%%%%%%%%%%%%

%%%%%%%%%%%%%%%%%%%%%%%%%%%%%%%%%%%%%%%%%%%%%%%%%%%%%%%%%%%%%%%%%%%%%%%%%%%%%%%%%%%%%%%%%%%%%%%%%%%%%%%%%%%%%%%%%%%%%%%%%%%%%%%%%%%%%%%%  BIBLIOGRAPHY %%%%%%%%%%%%%%%%%%%%%%%%%%%%%%%%%%%%%%%%%%%%%%%%%%%%%%%%%%%%%%%%%%%%%%

\bigskip 

\noindent{\tt Acknowledgements}. We thank M. Borovoi, F. Catanese, J.-L. Colliot-Th\'el\`ene, P. D\`ebes, A. Fehm, M. Fried, D. Harari, P. Oliverio  and P. Sarnak for their interest and important comments. %, and P. Oliverio for references related to Enriques surfaces. 
We also acknowledge the support of the ERC grant `Diophantine Problems' in this research, and the University of Konstanz for invitation to a school where the present paper started to be conceived.

\vfill

Pietro Corvaja \hfill Umberto Zannier

Universit\`a di Udine\hfill Scuola Normale Superiore

Via delle Scienze  \hfill Piazza dei Cavalieri, 7 

33100 Udine -ITALY \hfill 56126 Pisa - ITALY

pietro.corvaja@uniud.it\hfill u.zannier@sns.it

\end{document}